\documentclass[12pt,a4paper,openright]{article}
\usepackage[a4paper,textwidth=180mm]{geometry}
\usepackage[OT1]{fontenc}
\usepackage[utf8]{inputenc}
\usepackage{xcolor,lastpage,amsmath,amsthm,amssymb,mathrsfs,enumerate,graphicx,makeidx,hyperref,indentfirst,cite,color,microtype,mathtools,color,titlefoot}
\theoremstyle{definition}
\newtheorem{defn}{Definition}[section]

\newtheorem{prop}[defn]{Proposition}
\newtheorem{lema}[defn]{Lemma}
\newtheorem{teo}[defn]{Theorem}
\newtheorem{cor}[defn]{Corollary}
\newtheorem{obs}[defn]{Remark}

\DeclareMathOperator{\inte}{int}

\DeclareMathOperator{\dive}{div}
\DeclareMathOperator{\tr}{tr}

\DeclareMathOperator{\hess}{Hess}

\DeclareMathOperator{\id}{id}

\DeclareMathOperator{\Isom}{Isom}
\DeclareMathOperator{\re}{Re}
\DeclareMathOperator{\sn}{sn}
\DeclareMathOperator{\cn}{cn}
\title{Stable hypersurfaces with constant higher order mean curvature}
\author{Leonardo Damasceno\thanks{Instituto de Matemática, Universidade Federal do Rio de Janeiro, CP 68530, CEP 21941-909, Rio de Janeiro, Brazil. E-mail: \texttt{damasceno@im.ufrj.br}}\and Maria Fernanda Elbert\thanks{Instituto de Matemática, Universidade Federal do Rio de Janeiro, CP 68530, CEP 21941-909, Rio de Janeiro, Brazil. E-mail: \texttt{fernanda@im.ufrj.br}}}
\everymath{\displaystyle}
\allowdisplaybreaks
\begin{document}
\maketitle

\begin{abstract}
We propose a notion of stability for capillary hypersurfaces with constant higher order mean curvature and we generalize some results of the classical stability theory for CMC capillary hypersurfaces. 
\end{abstract}

\unmarkedfntext{Keywords: Isometric immersions, free boundary, capillary, higher order mean curvature}
\unmarkedfntext{2000 Mathematical Subject Classification: 53C42, 53A10.}
\unmarkedfntext{The authors were partially supported by CAPES.}

\section{Introduction}
\noindent Let $M^{n+1}$ be an oriented Riemannian manifold and $\varphi : \Sigma^n \rightarrow M$ be an oriented hypersurface. The higher order mean curvature of order $k$ of the hypersurface, $H_k$, is defined as the normalized $k$-th symmetric function of the principal curvatures of $\varphi : \Sigma^n \rightarrow M$. We recall that $H=H_1$ is the mean curvature. Consider a closed domain with smooth boundary $\Omega\subset M$.
In analogy with the constant mean curvature case (CMC), we say that a compact $H_k$-hypersurface in $\Omega$ is capillary if it is a $H_k$-hypersurface with non-empty boundary, whose boundary meets the $\partial \Omega$ at a constant angle. In the particular case where the angle is $\pi/2$, the capillary $H_k$-hypersurface is called free boundary. In this paper, we are interested in stability questions involving capillary $H_k$-hypersurfaces in $\Omega$.

The stability of the minimal or CMC hypersurfaces has been drawing the attention of many mathematicians over the years. Stable hypersurfaces in this case are those for which the second variation of the area is non-negative for a suitable class of variational problems. The existence or not of a boundary $\partial\Sigma$ on $\Sigma$ and the specific constraints on the variations determine the kind of stability problem we deal with. The stability results for $H_1$-hypersurfaces in the literature covers from that applied to closed (without boundary) hypersurfaces to that applied to capillary hypersurfaces, which includes the free boundary hypersurfaces as a particular case (see \cite{barbosa1984stability, lucas1988stability, ros1997stability, souam1997stability, wang2019uniqueness}).

We recall that a closed $H_k$-hypersurface, $k>1$, of a space form is a critical point for a modified area functional for volume preserving variations and it inherits the concept of stability from this variational problem (see \cite{marques1997stability}). We point out that a similar variational characterization of $H_k$-hypersurfaces in a general Riemannian manifold is not known. Despite this, in \cite{elbert2019note}, the second author and B. Nelli generalized the notion of stability for closed (or fixed boundary) $H_k$-hypersurfaces of a general Riemannian manifold for $k>1$.

Inspired by the approach of Alaee, Lesourd and Yau in \cite{alaee2021stable} when dealing capillary marginally outer trapped surfaces (MOTS), we propose a stability theory for capillary $H_k$-hypersurfaces, $k>1$, that generalizes the capillary CMC ($k=1$) stability theory (see \cite{ros1997stability, ros1995stability, souam1997stability}). We also notice that when the hypersurface has empty or fixed boundary, we recover the stability theory proposed in \cite{elbert2019note}, and a fortiori, the classical definition of stability for a CMC hypersurface with empty or fixed.

It is worthwhile to say that for $H_k$-hypersurfaces, $k>1$, we give two different notions of stability, the $k$-stability and the symmetric $k$-stability, and that for (simply connected) space forms these two notions coincide. 

The paper is organized as follows. The definition of $k$-stability is given in Section \ref{tres}. In Section \ref{quatro}, we deal with capillary $H_k$-hypersurfaces of space forms and we prove:

\ 
 
{\it Capillary geodesic caps (see Section \ref{quatro} for the precise definition) supported on a totally umbilical hypersurface $\partial\Omega$ of a space form are $k$-stable.}

\ 

In Section \ref{cinco}, we give some results for free boundary $H_k$-hypersurfaces supported on a geodesic sphere of a space form when $H_k=0$ and on a totally geodesic hypersurface of a space form. In Section \ref{seis}, we discuss the symmetric $k$-stability theory, and finally, in Section \ref{sete}, we address the stability of cylinders in $M \times \mathbb{R}$.

It is natural to ask if it is possible to generalize other stability results, established for $H_1$-hypersurfaces, for $H_k$-hypersurfaces. We quote, in particular, some results for capillary $H_1$-hypersurfaces supported on special hypersurfaces $\partial\Omega$ of space forms, where $\partial\Omega$ is:
\begin{itemize}
\item a geodesic sphere \cite{wang2019uniqueness};
\item a hyperplane or a slab \cite{ainouz2016stable, souam2021stable};
\item a wedge \cite{lopez2014capillary, choe2015stable, pyo2019rigidity};
\item the boundary of a region bounded by a union of hyperplanes of $\mathbb{R}^{n+1}$ \cite{li2017stability, souam2021stable2}.
\end{itemize}
For some of these questions we have partial results and we hope we can address them in a forthcoming work.

\section{Preliminaries}
Let $\left(M^{n+1},g\right)$ be an oriented Riemannian manifold and $\varphi : \Sigma^n \rightarrow M$ be an oriented hypersurface with unit normal vector field $\eta$ in the normal bundle $N\Sigma$. Its second fundamental form $\emph{II}$, scalar second fundamental form $\emph{II}_\eta$ and Weingarten operator $A=\left(\emph{II}_\eta\right)^\flat$ are defined, respectively, as
\begin{eqnarray*}
\emph{II}\left(X,Y\right) &=& \left(\overline\nabla_XY\right)^\perp=\left<\overline\nabla_XY,\eta\right>\eta=\emph{II}_\eta\left(X,Y\right)\eta \\
\left<A(X),Y\right> &=& \emph{II}_\eta\left(X,Y\right)=\left<-\overline\nabla_X\eta,Y\right>,
\end{eqnarray*}
where $X,Y \in \Gamma(T\Sigma)$ and $\overline\nabla $ is the Levi-Civita connection of $M$. Let $\left\{\kappa_1(p),...,\kappa_n(p)\right\}$ be the principal curvatures of the hypersurface $\varphi$ at $p$.

Let $\sigma_r$ be the $r$-th symmetric elementary polynomial
\begin{equation*}
\sigma_r(x_1,...,x_n)=\begin{dcases*}
1, & \quad $r=0$ \\
\sum_{1 \leq i_1<...<i_r \leq n} x_{i_1}...x_{i_r}, & \quad $r \in \{1,...,n\}$ \\
0, & \quad $r >n.$
\end{dcases*}.\end{equation*}
The normalized mean curvature of order $r$, $H_r$, and the non-normalized mean curvature of order $r$, $S_r$, of $\Sigma$ in $p$ are defined by
\begin{equation*}
H_r(p)=\binom{n}{r}^{-1}S_r(p)=\binom{n}{r}^{-1}\sigma_r(\kappa_1(p),...,\kappa_n(p)).
\end{equation*}
The functions $S_r$ can also be deduced from the equation
\begin{equation}\label{1}
\det\left(tI-A\right)=\sum_{r=0}^n \left(-1\right)^rS_rt^{n-r}.
\end{equation}
The hypersurface is said to have constant mean curvature of order $r$ if $H_r$ is constant over $\Sigma$; when this happens, $\Sigma$ is called an \textbf{$H_r$-hypersurface}.

The Newton transformations $P_r$ associated to the $r$-mean curvatures $H_r$ are defined by
\begin{equation*}
P_r=\begin{dcases*}
I, & \quad $r=0$ \\
S_rI-AP_{r-1}, & \quad $r \geq 1$
\end{dcases*}.\end{equation*}
Since $A_p$ is self-adjoint for all $p \in \Sigma$, the Newton transformations are self-adjoint as well and their eigenvectors are the same as those of $A$. Also, since \eqref{1} holds, it follows from the Cayley-Hamilton Theorem that $P_n=0$. If $\{e_1,...,e_n\}$ denotes the eigenvectors of $A$, let $S_r\left(A_i\right)$ be the symmetric elementary polynomial of order $r$ associated to $A_i=A\vert_{e_i^\perp}$. Then we have the following properties, whose proof can be found in \cite[Lemma 2.1]{marques1997stability}:
\begin{lema}\label{002.1}
For each $r \in \{1,...,n-1\}$,
\begin{enumerate}[(i)]
\item $P_re_i=S_r(A_i)e_i$
\item $\tr\,(P_r)=\sum_{i=1}^n S_r(A_i)=(n-r)S_r$
\item $\tr\,(P_rA)=\sum_{i=1}^n \kappa_iS_r(A_i)=(r+1)S_{r+1}$
\item $\tr\,(P_rA^2)=\sum_{i=1}^n \kappa_i^2S_r(A_i)=S_1S_{r+1}-(r+2)S_{r+2}.$
\end{enumerate}
\end{lema}

The next Lemma provide some useful inequalities involving the functions $H_r$.
\begin{lema}\label{002.2}
Suppose that $\kappa_1,...,\kappa_n \geq 0$. Then
\begin{enumerate}[(i)]
\item $H_{r-1}H_{r+1} \leq H_r^2$, $r \in \{1,...,n-1\}$
\item $H_1 \geq H_2^{1/2} \geq ... \geq H_n^{1/n}$
\item $H_1H_{r+1} \geq H_{r+2}$, $r \in \{0,...,n-2\}$
\end{enumerate}
and the previous inequalities are identities if and only if $\kappa_1=...=\kappa_n$.
\end{lema}
The proofs of (i) and (ii) can be found on \cite[p. 52]{hardy1952inequalities}, whereas (iii) is a direct consequence of (ii).

\

In a general Riemannian manifold $(M,g)$ with Levi-Civita connection $\overline\nabla$, if $\phi$ is a pointwise symmetric $(2,0)$-tensor in $M$, the Cheng-Yau operator of $f \in C^\infty(M)$ is defined by
\begin{equation*}
\Box f=\tr\left(\phi\left(\hess f\right)^\flat\right),
\end{equation*}
where $\hess f$ is the Hessian of $f$ in $M$ and $\left(\hess f\right)^\flat$ is the metric $(1,1)$-tensor field on $M$ equivalent to $\hess f$. After some basic computations the Cheng-Yau operator can be written as $$\Box f=\dive\left(\phi\overline\nabla f\right)-\left<\dive\phi,\overline\nabla f\right>,$$ where $\dive\phi:=\tr\left(\overline\nabla\phi\right)$. The operator $\phi$ is said to be divergence free if $\dive\phi=0$. In \cite[Theorem 4.1]{rosenberg1993hypersurfaces}, H. Rosenberg proved that $P_r$ is divergence free when $M$ has constant sectional curvature (see also \cite[Corollary 3.7]{elbert2002constant} for the case where $r=1$ and $M$ is Einstein). 

When considering an oriented hypersurface $\varphi : \Sigma \rightarrow M$ with shape operator $A$, the $L_r$-operator of $\Sigma$ is defined as the Cheng-Yau operator for the Newton transformation $P_r$, i.e.,
\begin{equation*}
L_rf=\tr\left(P_r\left(\hess f\right)^\flat\right), \quad f \in C^\infty(\Sigma).
\end{equation*}
Here, we say that $-L_r$ is a second-order elliptic differential operator when $P_r$ is positive definite on each point of $\Sigma$.

The next Lemma provides conditions that guarantee the ellipticity of $L_r$.
\begin{lema}\label{002.3}
With the same notation:
\begin{enumerate}[(i)]
\item If $H_2>0$ then, after a choice of orientation on $\varphi$, $P_1$ is positive definite.
\item If $H_{r+1}>0$ for $r>1$ and if there exists a point $p_0 \in \Sigma$ such that every principal curvature of $\Sigma$ in $p_0$ is positive then $P_j$ is positive definite for all $1 \leq j \leq r$.
\item If $H_{r+1}=0$ and the rank of $A$ is greater than $r$ then $P_r$ is definite.
\end{enumerate}
\end{lema}
The proofs of these statements can be found in \cite[Lemma 3.10]{elbert2002constant}, \cite[Proposition 3.2]{cheng2005embedded} and \cite[Corollary 2.3]{Hounie1995maximum}, respectively.
\bigskip

Let $\Omega \subseteq M$ be a closed domain with smooth boundary $\partial\Omega$ and assume that $\varphi : \Sigma \rightarrow M$ is an oriented hypersurface such that $\varphi(\Sigma) \subseteq \Omega$ and $\varphi(\partial\Sigma) \subseteq \partial\Omega$. Let $\nu \in \Gamma\left(T\Sigma\vert_{\partial\Sigma}\right)$ be the unit outward conormal vector field on $\partial\Sigma$ and let $\overline\nu \in \Gamma\left(T\partial\Omega\vert_{\partial\Sigma}\right)$ and $\overline\eta \in \Gamma\left(TM\vert_{\partial\Omega}\right)$ be the unit normal vector fields associated to the immersions $\varphi\vert_{\partial\Sigma} : \partial\Sigma \rightarrow \partial\Omega$ and $\iota_{\partial\Omega} : \partial\Omega \hookrightarrow M$, respectively, such that $\left\{\nu,\eta\right\}$ has the same orientation as $\left\{\overline\nu,\overline\eta\right\}$ on each point of $\varphi(\partial\Sigma)$. If $\theta$ denotes the angle between $\nu$ and $\overline\nu$, then
\begin{equation}\label{0002}\begin{dcases*}
\nu=\cos\theta\,\overline\nu+\sin\theta\,\overline\eta \\
\eta=-\sin\theta\,\overline\nu+\cos\theta\,\overline\eta
\end{dcases*}.\end{equation}
or conversely,
\begin{equation}\label{0092}\begin{dcases*}
\overline\nu=\cos\theta\,\nu-\sin\theta\,\eta \\
\overline\eta=\sin\theta\,\nu+\cos\theta\,\eta
\end{dcases*}.
\end{equation}
A $H_r$-hypersurface $\varphi : \Sigma \rightarrow \Omega \subseteq M$ is said to be \textbf{capillary} if the contact angle $\theta$ between $\partial\Sigma$ and $\partial\Omega$ is constant. When $\theta=\frac{\pi}{2}$, $\varphi$ is called a \textbf{free boundary hypersurface}. When $\varphi : \Sigma \rightarrow \Omega \subseteq M$ is a hypersurface with free boundary in $\Omega$, \eqref{0002} implies that $\nu=\overline\eta$ and $\eta=-\overline\nu$.

The following result will be used throughout this article.

\begin{lema}[{\cite[Lemma 2.2]{ainouz2016stable}}]\label{002.4}
Suppose $\iota_{\partial\Omega}$ is a totally umbilical immersion into $M$ and that $\varphi$ is a capillary $H_r$-hypersurface into $M$. Then the unit outwards normal vector field $\nu \in \Gamma\left(T\Sigma\vert_{\partial\Sigma}\right)$ is a principal direction of $\varphi$.
\end{lema}



\section{Stability of capillary $H_{r+1}$-hypersurfaces}\label{tres}
The notion of stability for $H_{r+1}$-hypersurfaces of a Riemannian manifold with empty (or fixed) boundary is well stated in the literature. The case $r=0$ (mean curvature) was established more than 40 years ago (see \cite{barbosa1984stability, lucas1988stability}). In 1997, L. Barbosa and G. Colares \cite{marques1997stability} addressed the empty (or fixed) boundary case for $H_{r+1}$-hypersurfaces of space forms and $r>0$. More than 20 years later, in a recent paper \cite{elbert2019note}, the second author and B. Nelli generalized this notion of stability for $H_{r+1}$-hypersurfaces of a general Riemannian manifold for $r>0$.

 For the capillary case the picture is more complex, since stability is settled only for mean curvature hypersurfaces ($r=0$) of space forms (see \cite{ros1995stability, souam1997stability}). The natural generalization of the variational problem for more general ambient spaces and $r>0$ has not made progress since the late 1990s. Inspired by the approach of \cite{alaee2021stable} we propose a notion of stability for capillary $H_{r+1}$-hypersurfaces that generalizes the works \cite{ros1995stability} and \cite{ souam1997stability} in these two directions: we deal with $r>0$ and with a more general ambient space. 

A variation of $\varphi$ is a smooth function $\Phi : \Sigma \times (-\varepsilon,\varepsilon) \rightarrow M$ such that, for each $t \in (-\varepsilon,\varepsilon)$, $\varphi_t=\Phi\vert_{\Sigma \times \{t\}}$ is an isometric immersion and $\varphi_0=\varphi$. The pair $(\Sigma,\varphi_t^*g)$ will be denoted by $\Sigma_t$. The variational field of $\Phi$ in $\varphi_t$ is defined by $$\xi_t(p)=\left.\Phi_*\frac{\partial}{\partial t}\right\vert_{(p,t)} \in \Gamma\left(TM\vert_{\varphi_t(\Sigma)}\right).$$ If $\eta_t \in \Gamma(N\Sigma)$ is the unit normal vector field of $\varphi_t$, the support function of $\Phi$ at $t$ is defined by $$f_t=\left<\xi_t,\eta_t\right> \in C^\infty(\Sigma).$$ Since $\varphi_t : \Sigma \rightarrow M$ is an oriented hypersurface, one can define its second fundamental form $\emph{II}_t$, its scalar second fundamental form $\left(\emph{II}_t\right)_{\eta_t}$ and its Weingarten operator $A_t$. 
Also, we set $\overline{R}_{\eta_t}\left(X\right):=\overline{R}\left(\eta_t,X\right)\eta_t$, where $\overline{R}$ the Riemann curvature tensor of $M$ defined by $$\overline{R}(X,Y)Z=\overline\nabla_Y\overline\nabla_XZ-\overline\nabla_X\overline\nabla_YZ+\overline\nabla_{[X,Y]}Z, \quad X,Y,Z \in \Gamma(TM).$$ If $S_{r+1}(t)$ denotes the non-normalized mean curvature of $(r+1)$-th order associated to immersion $\varphi_t$, its variation is given by
\begin{equation}\label{0004}
S_{r+1}^\prime(t)=\left(L_r\right)_tf_t+\left(S_1(t)S_{r+1}(t)-(r+2)S_{r+2}(t)\right)f_t+\tr_{\Sigma_t}\left(\left(P_r\overline{R}_\eta\right)_t\right)f_t+\xi_t^\top\left(S_{r+1}(t)\right),
\end{equation}
where $\left(L_r\right)_t$ is the $L_r$-operator of immersion $\varphi_t$ and $\left(P_r\overline{R}_\eta\right)_t:=\left(P_r\right)_t \circ \overline{R}_{\eta_t}$. A proof of \eqref{0004} can be found in \cite[Proposition 3.2]{elbert2002constant}.

The enclosed volume between $\Sigma$ and $\Sigma_t$ is defined as $\mathcal{V}(t)=\int_{\Sigma \times [0,t]} \Phi^*d\mu_M$, with $d\mu_M$ being the volume form of $(M,g)$. A variation $\Phi$ is volume-preserving if $\mathcal{V}(t)=\mathcal{V}(0)$ for all $t \in (-\varepsilon,\varepsilon)$. It is known that $$\mathcal{V}^\prime(0)=\int_\Sigma f\,d\mu_\Sigma,$$ where $u=\left<\xi,\eta\right> \in C^\infty(\Sigma)$ and $d\mu_\Sigma$ is the volume form of $\left(\Sigma,\varphi^*g\right)$. Thus, a variation $\Phi$ is volume-preserving if and only if $\int_\Sigma f\,d\mu_\Sigma=0$.

When dealing with stability questions we will be interested in the cases where $P_r$ is definite. For simplicity reasons, we will assume without loss of generality that the $H_{r+1}$-hypersurface has a positive definite Newton tensor for each point of $\Sigma$. With slight modifications, we can also address the case where $P_r$ is negative definite (see Remark \ref{004.5})

\begin{defn}\label{004.1}
We say that a $H_{r+1}$-hypersurface $\varphi : \Sigma \rightarrow M$ is \textbf{positive definite} if $P_r$ is positive definite on each point $p \in \Sigma$.
\end{defn}

A variation $\Phi$ of a hypersurface $\varphi : \Sigma \rightarrow \Omega \subseteq M$ is called admissible if $\varphi_t(\inte\Sigma) \subseteq \inte\Omega$ and $\varphi_t(\partial\Sigma) \subseteq \partial\Omega$ for any $t \in (-\varepsilon,\varepsilon)$, where $\varphi_t=\Phi\vert_{\Sigma \times \{t\}}$. If $\Phi$ is an admissible variation of $\varphi$, then $\xi\vert_{\partial\Sigma} \in \Gamma\left(T\partial\Omega\vert_{\partial\Sigma}\right)$. If $\Sigma$ is a capillary $H_{r+1}$-hypersurface supported on $\partial\Omega$ with contact angle $\theta \in (0,\pi)$ and $\Phi$ is a volume-preserving admissible variation of $\varphi$, define the functional
\begin{equation}\label{0005}
\mathcal{F}_{r,\theta}[\Sigma_t]=-\int_\Sigma S_{r+1}(t)\left<\xi_t,\eta_t\right>\,d\mu_{\Sigma_t}+\int_{\partial\Sigma} \left<\xi_t,(P_r\nu-\vert{P_r\nu}\vert\cos\theta\,\overline\nu)_t\right>\,d\mu_{\partial\Sigma_t},
\end{equation}
where $d\mu_{\Sigma_t}$ and $d\mu_{\partial\Sigma_t}$ denote the volume forms of $\Sigma_t$ and $\partial\Sigma_t=\left(\partial\Sigma,\left(\varphi_t\vert_{\partial\Sigma}\right)^*g\right)$, respectively. Note that when we set $r=0$, (\ref{0005}) is the first variation formula obtained in \cite{ros1995stability, souam1997stability}.

\bigskip

It can be proved (see \cite{lucas1988stability} and \cite{ros1995stability}) that for each smooth function $f$ on $\Sigma$ there exists an admissible normal variation of $\varphi$ with variation vector field $f\eta$. If, in addition, f satisfies $\int_\Sigma f\,d\mu_\Sigma=0$, the variation is volume-preserving.

In this article we will assume that $\partial\Omega$ is a totally umbilical hypersurface of $M$.

\begin{teo}\label{004.2}
If $\partial\Omega$ is totally umbilical and $\Phi$ is an admissible volume-preserving variation of a positive definite capillary $H_{r+1}$-hypersurface $\varphi : \Sigma\rightarrow \Omega \subseteq M$ supported on $\partial\Sigma$ then
\begin{multline}\label{0006}
\left.\frac{\partial}{\partial t}\mathcal{F}_{r,\theta}\left[\Sigma_t\right]\right\vert_{t=0}=-\int_\Sigma f\left(L_rf+\tr\left(P_r\left(A^2+\overline{R}_\eta\right)\right)f\right)\,d\mu_\Sigma+\\+\int_{\partial\Sigma} \left\vert{P_r\nu}\right\vert\,f\left(\frac{\partial f}{\partial\nu}+\left(\csc\theta\left(\emph{II}_{\partial\Omega}\right)_{\overline\eta}(\overline\nu,\overline\nu)-\cot\theta\left(\emph{II}_\Sigma\right)_\eta(\nu,\nu)\right)f\right)\,d\mu_{\partial\Sigma},
\end{multline}
where $f=\left<\xi,\eta\right> \in C^\infty(\Sigma)$ is the support function of $\Phi$ at $t=0$ and $\emph{II}_\Sigma$ and $\emph{II}_{\partial\Omega}$ are the second fundamental forms of $\varphi$ and $\iota_{\partial\Omega} : \partial\Omega \hookrightarrow \Omega$, respectively.
\end{teo}
A proof of Theorem \ref{004.2} will be given in the Appendix.

\begin{defn}\label{004.3}
A positive definite capillary $H_{r+1}$-hypersurface $\varphi : \Sigma \rightarrow \Omega \subseteq M$ supported on $\partial\Omega$ with contact angle $\theta \in (0,\pi)$ is \textbf{$r$-stable} if $\left.\frac{\partial}{\partial t}\mathcal{F}_{r,\theta}\left[\Sigma_t\right]\right\vert_{t=0} \geq 0$ for any volume-preserving admissible variation $\Phi$ of $\varphi$. If the inequality holds for all admissible variations of $\varphi$, $\Sigma$ is said to be \textbf{strongly $r$-stable}.
\end{defn}

Associated to \eqref{0006} we have the following eigenvalue problem:
\begin{equation}\label{0007}\begin{dcases*}
T_rf=-L_rf-q_rf=\lambda f, & \quad $\text{in}~\Sigma$ \\
\frac{\partial f}{\partial\nu}+\alpha_\theta f=0, & \quad $\text{on}~\partial\Sigma$
\end{dcases*},\end{equation}
where $q_r=\tr\left(P_r\left(A^2+\overline{R}_\eta\right)\right) \in C^\infty(\Sigma)$ and $\alpha_\theta=\csc\theta\left(\emph{II}_{\partial\Omega}\right)_{\overline\eta}(\overline\nu,\overline\nu)-\cot\theta\left(\emph{II}_\Sigma\right)_\eta(\nu,\nu) \in C^\infty(\partial\Sigma)$.

\ 

As we mentioned in the introduction, in a general ambient space, $T_r$ is not a "divergence form" operator and thus, \eqref{0006} gives rise to a non-symmetric bilinear form. Although $T_r$ is not self-adjoint, there exists a real eigenvalue $\lambda_p=\lambda_p(T_r)$ (see \cite[Theorem 3.1]{li2017eigenvalue}) called the \textbf{principal eigenvalue} such that any other eigenvalue $\lambda \in \mathbb{C}$ satisfies $\re\lambda \geq \lambda_p$. Also, the associated eigenspace has dimension equal to one and the associated eigenfunction is (strictly) positive.

\ 

The proposition below gives useful characterizations of $r$-stability.
\begin{prop}\label{004.4}
Let $\varphi : \Sigma \rightarrow \Omega \subseteq M$ be a positive definite capillary $H_{r+1}$-hypersurface supported on $\partial\Omega$. The following statements are equivalent:
\begin{enumerate}[(i)]
\item $\Sigma$ is strongly $r$-stable.
\item The principal eigenvalue $\lambda_p(T_r)$ of \eqref{0007} is non-negative.
\item There exists a positive function $f_0 \in C^\infty(\Sigma)$ such that $\frac{\partial f_0}{\partial\nu}+\alpha_\theta f_0=0$ on $\partial\Sigma$ and $T_rf_0 \geq 0$ in $\Sigma$.
\end{enumerate}
\end{prop}

\begin{proof}
Assume (i) true and let $f_0 \in C^\infty(\Sigma)$ be a (strictly) positive eigenfunction of \eqref{0007} associated to $\lambda_p$. Then $$\lambda_p(T_r)\int_\Sigma f_0^2\,d\mu_\Sigma=\int_\Sigma f_0T_rf_0\,d\mu_\Sigma+\int_{\partial\Sigma} \left\vert{P_r\nu}\right\vert f_0\left(\frac{\partial f_0}{\partial\nu}+\alpha_\theta f_0\right)\,d\mu_{\partial\Sigma} \geq 0,$$ 
which proves (ii). 

Now assume (ii) is true and let $f_0>0$ denotes an eigenfunction associated to the principal eigenvalue $\lambda_p$ of \eqref{0007}. Thus $T_rf_0=\lambda_pf_0 \geq 0$ in $\Sigma$ and this proves (iii).

Finally assume (iii) is true and let $f_0 \in C^\infty(\Sigma)$ be a positive function such that $\frac{\partial f_0}{\partial\nu}+\alpha_\theta f_0=0$ on $\partial\Sigma$ and $T_rf_0=\lambda_pf_0 \geq 0$ in $\Sigma$. Now, take $f \in C^\infty(\Sigma)$ and set $\tilde{f}=\frac{f}{f_0} \in C^\infty(\Sigma)$. For $f=\tilde{f}f_0$ we take an admissible normal variation of $\varphi$ with variation vector field $f\eta$. Then by using (\ref{0006}) and Lemma \ref{002.4} we obtain
\begin{eqnarray*}
\left.\frac{\partial}{\partial t}\mathcal{F}_{r,\theta}\left[\Sigma_t\right]\right\vert_{t=0} &=& -\int_\Sigma f\left(L_rf+q_rf\right)\,d\mu_\Sigma+\int_{\partial\Sigma}\left\vert{P_r\nu}\right\vert f\left(\frac{\partial f}{\partial\nu}+\alpha_\theta f\right)\,d\mu_{\partial\Sigma} \\
&=& -\int_\Sigma \tilde{f}f_0\left(L_r(\tilde{f}f_0)+q_r\tilde{f}f_0\right)\,d\mu_\Sigma+\int_{\partial\Sigma}\left\vert{P_r\nu}\right\vert \tilde{f}f_0\left(\frac{\partial (\tilde{f}f_0)}{\partial\nu}+\alpha_\theta\tilde{f}f_0\right)\,d\mu_{\partial\Sigma} \\
&=& -\int_\Sigma \tilde{f}^2f_0\left(L_rf_0+q_rf_0\right)+2\tilde{f}f_0\left<P_r\nabla\tilde{f},\nabla f_0\right>+\tilde{f}f_0^2L_r\tilde{f}\,d\mu_\Sigma+\\
&& +\int_{\partial\Sigma}\left\vert{P_r\nu}\right\vert \tilde{f}f_0\left(\tilde{f}\left(\frac{\partial f_0}{\partial\nu}+\alpha_\theta f_0\right)+f_0\frac{\partial\tilde{f}}{\partial\nu}\right)\,d\mu_{\partial\Sigma} \\
&\geq& -\int_\Sigma \tilde{f}f_0^2L_r\tilde{f}+2\tilde{f}f_0\left<P_r\nabla\tilde{f},\nabla f_0\right>\,d\mu_\Sigma+\int_{\partial\Sigma} \left<\tilde{f}f_0^2\nabla\tilde{f},P_r\nu\right>\,d\mu_{\partial\Sigma} \\
&=& \int_\Sigma \dive\left(\tilde{f}f_0^2P_r\nabla\tilde{f}\right)-\tilde{f}f_0^2L_r\tilde{f}-2\tilde{f}f_0\left<P_r\nabla\tilde{f},\nabla f_0\right>\,d\mu_\Sigma \\
&=& \int_\Sigma \tilde{f}f_0^2L_r\tilde{f}+\left<P_r\nabla\tilde{f},\nabla\left(\tilde{f}f_0^2\right)\right>-\tilde{f}f_0^2L_r\tilde{f}-2\tilde{f}f_0\left<P_r\nabla\tilde{f},\nabla f_0\right>\,d\mu_\Sigma \\
&=& \int_\Sigma f_0^2\left<P_r\nabla\tilde{f},\nabla\tilde{f}\right>\,d\mu_\Sigma \geq 0. \qedhere
\end{eqnarray*}
\end{proof}

\begin{obs}\label{004.5}
A similar construction can be made when considering $P_r$ being negative definite on each point of $\Sigma$. In this case the functional is defined to be $$\mathcal{F}_{r,\theta}[\Sigma_t]=\int_\Sigma S_{r+1}(t)\left<\xi_t,\eta_t\right>\,d\mu_{\Sigma_t}+\int_{\partial\Sigma} \left<\xi_t,(P_r\nu-\vert{P_r\nu}\vert\cos\theta\,\overline\nu)_t\right>\,d\mu_{\partial\Sigma_t}$$ and $$\left.\frac{\partial}{\partial t}\mathcal{F}_{r,\theta}\left[\Sigma_t\right]\right\vert_{t=0}=\int_\Sigma fT_rf\,d\mu_\Sigma+\int_{\partial\Sigma} \left\vert{P_r\nu}\right\vert\,f\left(\frac{\partial f}{\partial\nu}+\alpha_\theta f\right)\,d\mu_{\partial\Sigma},$$ where in this case $T_r=L_r+\tr\left(P_r\left(A^2+\overline{R}_\eta\right)\right)$.
\end{obs}

\section{Stability of capillary $H_{r+1}$-hypersurfaces in space forms}\label{quatro}

In this section we specialize to stability of $H_{r+1}$-hypersurfaces immersed on space forms and furnish a particular view of the theory for this setting.
In this paper, $\mathbb{M}^{n+1}(c)$ denote the simply connected space form of constant sectional curvature $c$, i.e., $\mathbb{M}^{n+1}(c)$ is equal to $\mathbb{R}^{n+1}$ if $c=0$, $\mathbb{S}^{n+1}(c)$ if $c>0$ and $\mathbb{H}^{n+1}(c)$ if $c=0$.

 As we mentioned before, in \cite[Theorem 4.1]{rosenberg1993hypersurfaces}, H. Rosenberg proved that $P_r$ is divergence free for hypersurfaces on space forms, i.e., $$\dive P_r=\sum_{i=1}^n (\nabla_{e_i}P_r)e_i=0,$$ where $\left\{e_i\right\}$ is an arbitrary local frame in $\Sigma$. Thus, the Stokes theorem and Lemma \ref{002.4} imply that 
\begin{equation}\label{0008}
-\int_\Sigma fL_rf\,d\mu_\Sigma=\int_\Sigma \left<P_r\nabla f,\nabla f\right>\,d\mu_\Sigma-\int_{\partial\Sigma}\left\vert{P_r\nu}\right\vert f\frac{\partial f}{\partial\nu}\,d\mu_{\partial\Sigma}
\end{equation}
for all $f \in C^\infty(\Sigma)$. Together with the identity $$\overline{R}(X,Y)Z=c\left(\left<X,Z\right>Y-\left<Y,Z\right>X\right), \quad X,Y,Z \in \Gamma\left(T\mathbb{M}^{n+1}(c)\right),$$ we have that $\overline{R}_\eta(X)=cX$ and $\tr\left(P_r\overline{R}_\eta\right)=c\tr P_r=c(n-r)S_r$. Hence, when the ambient space is a space form, \eqref{0006} becomes
\begin{multline}\label{0009}
\left.\frac{\partial}{\partial t}\mathcal{F}_{r,\theta}\left[\Sigma_t\right]\right\vert_{t=0}=\int_\Sigma \left<P_r\nabla f,\nabla f\right>-\left(S_1S_{r+1}-(r+2)S_{r+2}\right)f^2-c\left(n-r\right)S_rf^2\,d\mu_\Sigma+\\ +\int_{\partial\Sigma}\left\vert{P_r\nu}\right\vert\alpha_\theta f^2\,d\mu_{\partial\Sigma}
\end{multline}

Given a function $f \in C^\infty(\Sigma)$, the $H^1$-norm of $f$ is defined by $$\left\Vert{f}\right\Vert_{H^1(\Sigma)}^2=\left\Vert{f}\right\Vert_{L^2(\Sigma)}^2+\left\Vert{\nabla f}\right\Vert_{L^2(\Sigma)}^2=\int_\Sigma f^2+\left\vert\nabla f\right\vert^2\,d\mu_\Sigma$$ and the Sobolev space $H^1(\Sigma)$ is defined as the closure of $C^\infty(\Sigma)$ with respect to the norm $\left\Vert\cdot\right\Vert_{H^1(\Sigma)}$. This set endowed with this norm is a Hilbert space, thus one can view \eqref{0009} as a quadratic form associated to a bilinear symmetric form on $H^1(\Sigma)$.

\begin{defn}\label{005.1}
Let $\varphi : \Sigma \rightarrow \Omega \subseteq \mathbb{M}^{n+1}(c)$ be a positive definite $H_{r+1}$-hypersurface. The \textbf{$r$-index form} of $\varphi$ is defined by
\begin{equation}\label{0090}
\mathcal{I}_{r,\theta}(f_1,f_2)=\int_\Sigma\left<P_r\nabla f_1,\nabla f_2\right>-\tr\left(P_r\left(A^2+\overline{R}_\eta\right)\right)f_1f_2\,d\mu_\Sigma+\int_{\partial\Sigma}\vert{P_r\nu}\vert\alpha_\theta f_1f_2\,d\mu_{\partial\Sigma}
\end{equation}
where $f_1,f_2 \in H^1(\Sigma)$, or equivalently,
\begin{multline}\label{0010}
\mathcal{I}_{r,\theta}(f_1,f_2)=\int_\Sigma \left<P_r\nabla f_1,\nabla f_2\right>-\left(S_1S_{r+1}-(r+2)S_{r+2}+c\left(n-r\right)S_r\right)f_1f_2\,d\mu_\Sigma+\\+\int_{\partial\Sigma}\left\vert{P_r\nu}\right\vert\alpha_\theta f_1f_2\,d\mu_{\partial\Sigma},
\end{multline}
\end{defn}

From the above definition, one can notice that $\Sigma$ is strongly $r$-stable if and only if $\mathcal{I}_{r,\theta}(f,f) \geq 0$ for all $f \in H^1(\Sigma)$ and that $\Sigma$ is $r$-stable if $\mathcal{I}_{r,\theta}(f,f) \geq 0$ for all $f \in \mathcal{F}=\left\{f \in H^1(\Sigma)\,|\,\int_\Sigma f\,d\mu_\Sigma=0\right\}$.

As in the case $r=0$, when considering $\varphi$ a capillary $(r+1)$-minimal hypersurface, i.e. $H_{r+1}=0$, we say that $\varphi$ is stable if $\mathcal{I}_{r,\theta}(f,f) \geq 0$ for all $f \in C_0^\infty(\Sigma)$. This means the hypothesis on the variation being volume-preserving is dropped.

\

It is known that if $\Sigma$ is a totally umbilical hypersurface of $\mathbb{M}^{n+1}(c)$ then $\Sigma$ is contained into a totally geodesic hypersurface, a geodesic sphere and, for $c<0$, $\Sigma$ can also be part of a horosphere or a equidistant hypersurface \cite[p. 75, 77]{spivak1970comprehensive}. Horospheres have all principal curvatures equal to $1$ and equidistant hypersurfaces have all principal curvatures positive and smaller than $1$. The next result, whose proof is given in Appendix B, shows that such totally umbilical hypersurfaces are examples of $r$-stable capillary $H_{r+1}$-hypersurfaces of $\mathbb{M}^{n+1}(c)$.

\begin{prop}\label{004.6}
Suppose that $\varphi : \Sigma \rightarrow \Omega \subseteq \mathbb{M}^{n+1}(c)$ is a positive definite compact totally umbilical capillary $H_{r+1}$-hypersurface supported on a connected totally umbilical hypersurface $\partial\Omega$ of $\mathbb{M}^{n+1}(c)$. Then $\varphi$ is $r$-stable.
\end{prop}

The $0$-stability of a totally umbilical hypersurface supported on a horosphere is discussed in \cite[Proposition 2.5]{guo2022stable}.

A \textbf{geodesic cap} $\Sigma$ of $\mathbb{M}^{n+1}(c)$ is a geodesic ball of a geodesic sphere of $\mathbb{M}^{n+1}(c)$. The proposition above shows such capillary hypersurfaces supported on totally umbilical hypersurface are $r$-stable.

\begin{cor}
Capillary geodesic caps supported on a totally umbilical hypersurface of $\mathbb{M}^{n+1}(c)$ are $r$-stable.
\end{cor}

\bigskip

A normal vector field $\xi=f\eta$, with $f \in \mathcal{F}$, is a \textbf{Jacobi field} if $f \in \ker\mathcal{I}_{r,\theta}\vert_{\mathcal{F} \times \mathcal{F}}$, i.e., $\mathcal{I}_{r,\theta}(f,g)=0$ for every $g \in \mathcal{F}$. The next lemma gives a characterization of Jacobi fields on $\Sigma$.

\begin{lema}\label{005.3}
Let $\varphi : \Sigma \rightarrow \Omega \subseteq \mathbb{M}^{n+1}(c)$ be a positive definite $H_{r+1}$-hypersurface with free boundary in $\partial\Omega$ and $f \in \mathcal{F}$. Then

\begin{enumerate}[(i)]
\item $\xi=f\eta$ is a Jacobi field on $\Sigma$ if and only if $f \in C^\infty(\Sigma)$ and
\begin{equation}\label{0011}\begin{dcases*}
T_rf=-L_rf-\tr\left(A^2\left(P_r+\overline{R}_\eta\right)\right)f=\text{constant} & in $\Sigma$ \\
\frac{\partial f}{\partial\nu}+\alpha_\theta f=0 & on $\partial\Sigma$
\end{dcases*}.\end{equation}

\item If $\varphi$ is $r$-stable and $\mathcal{I}_{r,\theta}(f,f)=0$ then $f$ is a Jacobi field on $\Sigma$.
\end{enumerate}
\end{lema}

\begin{proof}
\begin{enumerate}[(i)]
\item First suppose that $f \in C^\infty(\Sigma)$ satisfies \eqref{0011} and let $g \in \mathcal{F}$. If $\left(g_m\right)_{m \in \mathbb{N}}$ is a sequence of smooth functions on $\Sigma$ such that $\int_\Sigma g_m\,d\mu_\Sigma=0$ for all $m \in \mathbb{N}$ and $g_m \stackrel{m \rightarrow \infty}{\longrightarrow} g$ in the $H^1(\Sigma)$-sense, it follows from \eqref{0010} and \eqref{0011} that
\begin{eqnarray*}
\mathcal{I}_{r,\theta}(f,g_m) &=& -\int_\Sigma \left(L_rf+\tr_\Sigma\left(P_r\left(A^2+\overline{R}_\eta\right)\right)f\right)g_m\,d\mu_\Sigma+\int_{\partial\Sigma}\left\vert{P_r\nu}\right\vert\left(\frac{\partial f}{\partial\nu}+\alpha_\theta f\right)g_m\,d\mu_{\partial\Sigma} \\
&=& 0 \\
\mathcal{I}_{r,\theta}(f,g) &=& \lim_{m \rightarrow \infty} \mathcal{I}_{r,\theta}(f,g_m)=0,
\end{eqnarray*}
proving that $f\eta$ is a Jacobi field on $\Sigma$.

To prove the converse, we will first claim that
\begin{equation}\label{0012}
\mathcal{I}_{r,\theta}(f,g)=b\int_\Sigma g\,d\mu_\Sigma, \quad \forall g \in C_0^\infty(\Sigma),
\end{equation}
where $b$ is a constant to be specified. In fact, let $g_1 \in C_0^\infty(\Sigma)$ be a function such that $\int_\Sigma g_1\,d\mu_\Sigma \neq 0$ and define $b=\frac{\mathcal{I}_{r,\theta}(f,g_1)}{\int_\Sigma g_1\,d\mu_\Sigma}$. Given $g \in C^\infty(\Sigma)$, let $g_2=-\frac{\int_\Sigma g\,d\mu_\Sigma}{\int_\Sigma g_1\,d\mu_\Sigma}g_1+g \in \mathcal{F}$. Then,
\begin{equation*}
0=\mathcal{I}_{r,\theta}(f,g_2)=-\dfrac{\int_\Sigma g\,d\mu_\Sigma}{\int_\Sigma g_1\,d\mu_\Sigma}\mathcal{I}_{r,\theta}(f,g_1)+\mathcal{I}_{r,\theta}(f,g)=-b\int_\Sigma g\,d\mu_\Sigma+\mathcal{I}_{r,\theta}(f,g),
\end{equation*}
proving \eqref{0012}. Thus, $f$ is a weak solution to the first equation in \eqref{0011} and since $P_r$ is positive definite at each point of $\Sigma$, the regularity theory of second-order elliptic self-adjoint linear operator implies that $f \in C^\infty(\Sigma)$, proving that $f$ satisfies the equation in the strong sense.

In order to prove that $\frac{\partial f}{\partial\nu}+\alpha_\theta f=0$ on $\partial\Sigma$, let
\begin{equation*}g=\begin{dcases*}
\frac{\partial f}{\partial\nu}+\alpha_\theta f,& in $\partial\Sigma$ \\
0,& else
\end{dcases*}.\end{equation*}
Since $f\eta$ is a Jacobi field, we have
\begin{equation*}
0=\mathcal{I}_{r,\theta}(f,g)=\int_{\partial\Sigma}\left\vert{P_r\nu}\right\vert\left(\frac{\partial f}{\partial\nu}+\alpha_\theta f\right)^2\,d\mu_{\partial\Sigma},
\end{equation*}
proving (i).

\item Let $g \in \mathcal{F}$. Since $\varphi$ is $r$-stable, we have $\mathcal{I}_{r,\theta}\left(f+\varepsilon g,f+\varepsilon g\right) \geq 0$ for all $\varepsilon>0$. Thus,
\begin{eqnarray}
0 &\leq& \mathcal{I}_{r,\theta}(f+\varepsilon g,f+\varepsilon g)=\mathcal{I}_{r,\theta}(f,f)+2\varepsilon\,\mathcal{I}_{r,\theta}(f,g)+\varepsilon^2\mathcal{I}_{r,\theta}(g,g) \nonumber \\
&=& 2\varepsilon\,\mathcal{I}_{r,\theta}(f,g)+\varepsilon^2\mathcal{I}_{r,\theta}(g,g). \label{0013}
\end{eqnarray}
Dividing \eqref{0013} by $2\varepsilon$ and letting $\varepsilon \rightarrow 0^+$, we obtain $\mathcal{I}_{r,\theta}(f,g) \geq 0$. The reverse inequality is obtained replacing $\varepsilon$ by $-\varepsilon$. \qedhere
\end{enumerate}
\end{proof}

\section{Stability results for free boundary $H_{r+1}$-hypersurfaces in $\mathbb{M}^{n+1}(c)$}\label{cinco}
In this section we will assume that $\theta=\frac{\pi}{2}$ and we will use following models for $\mathbb{M}^{n+1}(c)$:
\begin{eqnarray*}
\mathbb{R}^{n+1} &=& \left\{x=(x_1,...,x_{n+2}) \in \mathbb{R}^{n+2}\,\left|\right.\,x_{n+2}=0\right\} \\
\mathbb{S}^{n+1} &=& \left\{x=(x_1,...,x_{n+2}) \in \mathbb{R}^{n+2}\,\left|\right.\,x_1^2+...+x_{n+2}^2=\frac{1}{c^2}\right\} \\
\mathbb{H}^{n+1}(c) &=& \left\{x=(x_1,...,x_{n+2}) \in \mathbb{R}_1^{n+2}\,\left|\right.\,x_1^2+...+x_{n+1}^2-x_{n+2}^2=-\frac{1}{c^2}, x_{n+2}>0\right\}
\end{eqnarray*}
endowed with the pullback of the Euclidean metric for $c \geq 0$ or the Minkowski metric for $c<0$. Define
\begin{equation}\label{0072}
\sn_c(\rho)=\begin{dcases*}
\frac{\sin\left(\rho\sqrt{c}\right)}{\sqrt{c}},& if $c>0$ \\
\rho,& if $c=0$ \\
\frac{\sinh\left(\rho\sqrt{-c}\right)}{\sqrt{-c}},& if $c<0$
\end{dcases*}
\end{equation}
and $\cn_c(\rho)=\sn_c^\prime(\rho)$. For the results of this section we will assume that $c \in \{-1,0,1\}$. If $\{\mathbf{e}_1,...,\mathbf{e}_{n+2}\}$ denote the vectors of the canonical basis of $\mathbb{R}^{n+2}$ then we have some relations involving $\Sigma$, $\nu$ and $\varphi$. The proof can be found in \cite[Lemma 1.1]{souam1997stability}.

\begin{lema}\label{007.2}
Let $B_R$ be a geodesic ball of $\mathbb{M}^{n+1}(c)$ and let $\varphi : \Sigma^n \rightarrow B_R \subseteq \mathbb{M}^{n+1}(c)$ be a $H_{r+1}$-hypersurface. If $\{\mathbf{e}_1,...,\mathbf{e}_{n+1},\mathbf{e}_{n+2}\}$ denote the vectors of the canonical basis in $\mathbb{R}^{n+2}$, we have
\begin{enumerate}[(i)]
\item On $\partial\Sigma$ we have $\sn_c(R)\,\nu-\cn_c(R)\,\varphi=-\frac{1}{\sqrt{c}}\,\mathbf{e}_{n+2}$ for $c \neq 0$ and $\nu=\frac{1}{R}\varphi$ for $c=0$.
\item The second fundamental form of $\iota_{\partial\Sigma} : \partial\Sigma \hookrightarrow \Sigma$ with respect to $-\nu$ is given by $\frac{\cn_c(R)}{\sn_c(R)}\left<\cdot,\cdot\right>$, where $\left<\cdot,\cdot\right>$ is the metric of $\partial\Sigma$ induced by $\varphi\vert_{\partial\Sigma}$. In particular, if $n=2$, the geodesic curvature of $\partial\Sigma$ in $\Sigma$ of any point is given by $\frac{\cn_c(R)}{\sn_c(R)}$.
\end{enumerate}
\end{lema}

The next Lemma gives important relations between $\varphi$ and $\eta$.

\begin{lema}\label{006.2}
Let $\varphi : \Sigma^n \rightarrow \mathbb{M}^{n+1}(c) \subseteq \mathbb{R}^{n+2}$ be a hypersurface. Then
\begin{eqnarray}
L_r\varphi &=& (r+1)S_{r+1}\eta-c(n-r)S_r\varphi \label{0015} \\
L_r\eta &=& -\tr\left(P_rA^2\right)\eta+c(r+1)S_{r+1}\varphi-\nabla S_{r+1}, \label{0016}
\end{eqnarray}
where $L_r\varphi$ and $L_r\eta$ are calculated coordinate-wise. Moreover, if $c=0$ then
\begin{eqnarray}
\frac{1}{2}L_r\left\vert\varphi\right\vert^2 &=& (n-r)S_r+(r+1)S_{r+1}\left<\varphi,\eta\right> \label{0017} \\
L_r\left<\varphi,\eta\right> &=& -(r+1)S_{r+1}-\left(S_1S_{r+1}-(r+2)S_{r+2}\right)\left<\varphi,\eta\right>-\left<\nabla S_{r+1},\varphi^\top\right> \label{0018}
\end{eqnarray}
\end{lema}
For a proof of \eqref{0015} and \eqref{0016} see \cite[Remark 5.1]{rosenberg1993hypersurfaces} and for a proof of \eqref{0017} and \eqref{0018} see \cite[Lemma 1, (b)]{alencar1998integral} and \cite[Lemma 2]{alencar1998integral}, respectively.

\begin{teo}\label{007.1}
Let $B_R \subseteq \mathbb{M}^{n+1}(c)$ be a geodesic ball with radius $R>0$. Then, for $r>0$, it does not exist a positive $r$-stable $(r+1)$-minimal hypersurface with free boundary in $B_R$.
\end{teo}

\begin{proof}
The proof of this result is based on \cite[Theorem 2.1]{souam1997stability}. Suppose such hypersurface $\varphi : \Sigma \rightarrow B_R \subseteq \mathbb{M}^{n+1}(c)$ exists and define the vector $\widetilde\varphi=\int_\Sigma \varphi\,d\mu_\Sigma$. One can take $n$ linearly independent vectors $\mathbf{u}_1,...,\mathbf{u}_n \in \mathbb{R}^{n+1}$ such that $\left<\mathbf{e}_{n+2},\mathbf{u}_i\right>=0$ and $$0=\left<\widetilde\varphi,\mathbf{u}_i\right>=\int_\Sigma \left<\varphi,\mathbf{u}_i\right>\,d\mu_\Sigma$$ for all $i \in \{1,...,n\}$. Let $f_i=\left<\varphi,\mathbf{u}_i\right> \in \mathcal{F}$ for $i \in \{1,...,n\}$. The formula \eqref{0015}, the hypothesis $S_{r+1}=0$ and the item (i) of Lemma \ref{007.2} yield to
\begin{eqnarray}
T_rf_i &=& -\left(L_r+\tr\left(P_rA^2\right)+c(n-r)S_r\right)f_i \nonumber \\
&=& -\tr\left(P_rA^2\right)f_i=\left\vert\sqrt{P_r}A\right\vert^2f_i, \label{0025}
\end{eqnarray}
and
\begin{equation}\label{0026}
\frac{\partial f_i}{\partial\nu}+\left(\emph{II}_{\partial B}\right)_{\overline\eta}\left(\overline\nu,\overline\nu\right)f_i=\nu\left<\varphi,\mathbf{u}_i\right>-\frac{\cn_c(R)}{\sn_c(R)}\left<\varphi,\mathbf{u}_i\right>=\left<\nu-\frac{\cn_c(R)}{\sn_c(R)}\varphi,\mathbf{u}_i\right>=0
\end{equation}
on $\partial\Sigma$, the equations \eqref{0025} and \eqref{0026} applied to the $r$-stability hypothesis imply that
\begin{eqnarray}
0 &\leq& \mathcal{I}_{r,\pi/2}(f_i,f_i)=\int_\Sigma f_iT_rf_i\,d\mu_\Sigma+\int_{\partial\Sigma}\left\vert{P_r\nu}\right\vert f_i\left(\frac{\partial f_i}{\partial\nu}+\left(\emph{II}_{\partial B}\right)_{\overline\eta}\left(\overline\nu,\overline\nu\right)f_i\right)\,d\mu_{\partial\Sigma} \nonumber \\
&=& -\int_\Sigma \left\vert\sqrt{P_r}A\right\vert^2f_i^2\,d\mu_\Sigma \label{0027}
\end{eqnarray}
for all $i \in \{1,...,n\}$. Summing up the inequality \eqref{0027} for all $i \in \{1,...,n\}$ we obtain $$0 \leq \sum_{i=1}^n \mathcal{I}_{r,\pi/2}(f_i,f_i)=-\int_\Sigma \left\vert\sqrt{P_r}A\right\vert^2\left(\sum_{i=1}^n f_i^2\right)\,d\mu_\Sigma \leq 0.$$ Thus $\left\vert\sqrt{P_r}A\right\vert^2\left(\sum_{i=1}^n f_i^2\right) \equiv 0$ on $\Sigma$. Notice that the functions $f_i \in \mathcal{F}$ are restrictions to $\Sigma$ of $n$ linearly independent linear forms $F_i=\left<\cdot,\mathbf{u}_i\right> \in \left(\mathbb{R}^{n+2}\right)^*$. Since $$\left(\sum_{i=1}^n f_i^2\right)^{-1}(\{0\})=\bigcap_{i=1}^n f_i^{-1}(\{0\})=\bigcap_{i=1}^n \left(\ker F_i \cap \Sigma\right)=\left(\bigcap_{i=1}^n \ker F_i\right) \cap \Sigma \subseteq \left(\bigcap_{i=1}^n \ker F_i\right),$$ the set $\left(\sum_{i=1}^n f_i^2\right)^{-1}(\{0\})$ is contained into a $2$-dimensional linear subspace whose intersection with $\mathbb{M}^{n+1}(c)$ has dimension, at most, equal to one. Thus $\left(\sum_{i=1}^n f_i^2\right)^{-1}(\{0\})$ has measure zero and, therefore, $\left\vert\sqrt{P_r}A\right\vert^2=0$ on $\Sigma$. Since $\sqrt{P_r}$ is also positive definite, it is invertible and $$A=\left(\sqrt{P_r}\right)^{-1}\sqrt{P_r}A=0,$$ proving that $\Sigma$ is totally geodesic and obtaining a contradiction since $P_r$ is assumed to be a definite operator on each point of $\Sigma$.
\end{proof}

Before proving the next theorem, we need the following definition and the following Lemma.

\begin{defn}\label{009.2}
The $k$-th \textbf{coefficient of umbilicity} $\tau_k$ of a hypersurface $\varphi : \Sigma \rightarrow M$ is defined by $$\tau_k=(k+1)^2S_{k+1}^2-(n-k)S_k\tr\left(P_kA^2\right).$$
\end{defn}

\begin{lema}[Umbilicity Lemma]\label{009.3}
If $\kappa_1,...,\kappa_n \geq 0$ and $P_k$ is positive definite on each point of $\Sigma$ then $$\tau_k=(k+1)^2S_{k+1}^2-(n-k)S_k\tr\left(P_kA^2\right) \leq 0$$ and $\tau_k \equiv 0$ if and only if $\kappa_1=...=\kappa_n$.
\end{lema}

Notice that for $k=0$, $\tau_0=S_1^2-nS_0\tr A^2=n\left(nH^2-\vert{A}\vert^2\right) \leq 0$.

\begin{proof}
Since $\kappa_1,...,\kappa_n\geq 0$ and $(n-k)\binom{n}{k}=(k+1)\binom{n}{k+1}$ for all $k \in \{0,...,n-1\}$, it follows from Lemma \ref{002.2} that
\begin{eqnarray*}
\tau_k &=& (k+1)^2S_{k+1}^2-(n-k)S_k\left(S_1S_{k+1}-(k+2)S_{k+2}\right) \\
&=& (k+1)^2\binom{n}{k+1}^2H_{k+1}^2-(n-k)\binom{n}{k}H_k\left(n\binom{n}{k+1}H_1H_{k+1}-(k+2)\binom{n}{k+2}H_{k+2}\right) \\
&=& (n-k)\binom{n}{k}\left((n-k)\binom{n}{k}H_{k+1}^2-H_k\left(\frac{n}{k+1}H_1H_{k+1}-(n-k-1)\binom{n}{k+1}H_{k+2}\right)\right) \\
&=& \left((n-k)\binom{n}{k}\right)^2\left(H_{k+1}^2-H_k\left(\frac{nH_1H_{k+1}-(n-k-1)H_{k+2}}{k+1}\right)\right) \\
&\leq& \left((n-k)\binom{n}{k}\right)^2\left(H_{k+1}^2-H_1H_kH_{k+1}\right) \leq 0.
\end{eqnarray*}
If the equality holds on $\Sigma$ then we have $$0=\tau_k \leq H_{k+1}\left((n-k)\binom{n}{k}\right)^2\left(H_{k+1}-H_1H_k\right) \leq 0,$$ which shows that $H_1H_k \equiv H_{k+1}$ on $\Sigma$ and it follows from Lemma \ref{002.2} that $\kappa_1=...=\kappa_n$ at each point of $\Sigma$.
\end{proof}

The next result generalizes \cite[Theorem 3.1]{souam1997stability} for any $r \in \{0,...,n-1\}$ and \cite[Theorem 5.1 (i)]{ainouz2016stable} for any stable $H_{r+1}$-hypersurface with free boundary in totally geodesic hypersurfaces of space forms.

\begin{teo}\label{009.1}
Let $\Pi$ be a totally geodesic hypersurface of $\mathbb{M}^{n+1}(c)$ and let $\varphi : \Sigma \rightarrow \mathbb{M}^{n+1}(c)$ be a compact $r$-stable hypersurface with free boundary in $\Pi$ which lies in one side of $\Pi$. Then $\varphi(\Sigma)$ is a geodesic half-sphere whose center is in $\Pi$.
\end{teo}

\begin{proof}
 First consider $c \in \{-1,0\}$ and, without loss of generality, assume that $\Pi=\widetilde\Pi \cap \mathbb{M}^{n+1}(c)$, where $\widetilde\Pi$ is the hyperplane of $\mathbb{R}^{n+2}$ with equation $x_1=0$. If $\phi \in \Isom(\mathbb{M}^{n+1}(c))$ is the isometry defined by $$\phi(x_1,x_2,...,x_{n+2})=(-x_1,x_2,...,x_{n+2}), \quad (x_1,...,x_{n+2}) \in \mathbb{M}^{n+1}(c),$$ then the image of $\Sigma$ through $\phi$ is also a $r$-stable $H_{r+1}$-hypersurface with free boundary in $\Pi$ since given any function $f_0 \in \mathcal{F}_{\phi(\Sigma)}=\left\{f \in H^1(\phi(\Sigma))\,\left|\right.\,\int_{\phi(\Sigma)}f\,d\mu_{\phi(\Sigma)}=0\right\}$, we have $f_0 \circ \phi\vert_\Sigma^{-1} \in \mathcal{F}_\Sigma$. Thus, $\mathcal{I}_r^{\phi(\Sigma)}(f_0,f_0)=\mathcal{I}_r^\Sigma(f_0\circ\phi\vert_\Sigma^{-1},f_0\circ\phi\vert_\Sigma^{-1}) \geq 0$, proving the claim. Now since $\widetilde\Sigma=\Sigma \cup \phi(\Sigma)$ is a closed $H_{r+1}$-hypersurface in $\mathbb{M}^{n+1}(c)$, the same argument used in \cite[Remark 5.2]{ainouz2016stable} gives that $\widetilde\Sigma$ is a geodesic sphere and $\Sigma$ is a geodesic half-sphere.

Now consider the case that $c=1$. Since \cite[Theorem 5.3]{marques1997stability} only holds for hypersurfaces contained in a hemisphere of $\mathbb{S}^{n+1}$, the approach used in the case of $c \leq 0$ cannot be used here. Let $\widetilde\varphi=\int_\Sigma\varphi\,d\mu_\Sigma$ and $\widetilde\eta=\int_\Sigma\eta\,d\mu_\Sigma$, seen as constant vectors in $\mathbb{R}^{n+2}$. Since the unit normal vector field of $\Pi \cong \mathbb{S}^n$ is the restriction of a constant vector field of $\mathbb{R}^{n+2}$, the free boundary condition implies that $\nu$ is equal to a constant vector field named $\mathbf{v} \in \mathbb{R}^{n+2}$. Thus one can find $n-1$ linearly independent vectors $\mathbf{u}_1,...,\mathbf{u}_{n-1} \in \mathbb{R}^{n+2}$ such that
\begin{equation}\label{0028}
\left<\widetilde\varphi,\mathbf{u}_i\right>=\left<\widetilde\eta,\mathbf{u}_i\right>=\left<\mathbf{v},\mathbf{u}_i\right>=0, \quad i \in \{1,...,n-1\}.
\end{equation}

For each $i \in \{1,...,n-1\}$ define the functions $f_i=\left<\varphi,\mathbf{u}_i\right>$, $g_i=\left<\eta,\mathbf{u}_i\right>$, $i \in \{1,...,n-1\}$, on $\Sigma$. From \eqref{0028}, we have that $f_i,g_i \in \mathcal{F} \cap C^\infty(\Sigma)$ for all $i \in \{1,...,n\}$. The equations \eqref{0015} and \eqref{0016} of Lemma \ref{006.2} gives for $i \in \{1,...,n-1\}$
\begin{eqnarray}
T_rf_i &=& -L_r\left<\varphi,\mathbf{u}_i\right>-\left(\tr\left(P_rA^2\right)+(n-r)S_r\right)\left<\varphi,\mathbf{u}_i\right> \nonumber \\
&=& -\left(r+1\right)S_{r+1}g_i-\tr\left(P_rA^2\right)f_i \label{0029}
\end{eqnarray}
and
\begin{eqnarray}
T_rg_i &=& -L_r\left<\eta,\mathbf{u}_i\right>-\left(\tr\left(P_rA^2\right)+(n-r)S_r\right)\left<\eta,\mathbf{u}_i\right> \nonumber \\
&=& -\left(r+1\right)S_{r+1}f_i-\left(n-r\right)S_rg_i. \label{0030}
\end{eqnarray}
Since Lemma \ref{002.4} implies that $\nu$ is a principal direction of $\Sigma$ along $\partial\Sigma$, for $i \in \{1,..,n-1\}$ we have on $\partial\Sigma$
\begin{eqnarray}
\frac{\partial f_i}{\partial\nu} &=& \nu\left<\varphi,\mathbf{u}_i\right>=\left<\nu,\mathbf{u}_i\right>=0 \label{0031} \\
\frac{\partial g_i}{\partial\nu} &=& \nu\left<\eta,\mathbf{u}_i\right>=\left<\widetilde\nabla_\nu\eta,\mathbf{u}_i\right>=\left<\overline\nabla_\nu\eta+\left<\nu,\eta\right>\varphi,\mathbf{u}_i\right>=-\left<A\nu,\mathbf{u}_i\right>=-\left\vert{A\nu}\right\vert\left<\nu,\mathbf{u}_i\right>=0, \label{0032}
\end{eqnarray}
where $\widetilde\nabla$ is the Levi-Civita connection of $\mathbb{R}^{n+2}$. By using \eqref{0029}, \eqref{0030}, \eqref{0031} and \eqref{0032} in \eqref{0010} we obtain for $i \in \{1,...,n-1\}$
\begin{eqnarray}
\mathcal{I}_{r,\pi/2}\left(f_i,f_i\right) &=& \int_\Sigma f_iT_rf_i\,d\mu_\Sigma+\int_{\partial\Sigma}\left\vert{P_r\nu}\right\vert f_i\frac{\partial f_i}{\partial\nu}\,d\mu_{\partial\Sigma} \nonumber \\
&=& -\int_\Sigma \tr\left(P_rA^2\right)f_i^2+\left(r+1\right)S_{r+1}f_ig_i\,d\mu_\Sigma \label{0033} \\
\mathcal{I}_{r,\pi/2}\left(g_i,g_i\right) &=& \int_\Sigma g_iT_rg_i\,d\mu_\Sigma+\int_{\partial\Sigma}\left\vert{P_r\nu}\right\vert g_i\frac{\partial g_i}{\partial\nu}\,d\mu_{\partial\Sigma} \nonumber \\
&=& -\int_\Sigma \left(r+1\right)S_{r+1}f_ig_i+\left(n-r\right)S_rg_i^2\,d\mu_\Sigma \label{0034}
\end{eqnarray}

The $r$-stability hypothesis applied to the equations \eqref{0033} and \eqref{0034} for all $i \in \{1,...,n-1\}$ yield to
\begin{eqnarray}
0 &\leq& \mathcal{I}_r\left(f_i,f_i\right)+\mathcal{I}_r\left(g_i,g_i\right) \nonumber \\
&=& -\int_\Sigma \tr\left(P_rA^2\right)f_i^2+2\left(r+1\right)S_{r+1}f_ig_i+\left(n-r\right)S_rg_i^2\,d\mu_\Sigma \nonumber \\
&=& -\int_\Sigma\frac{1}{(n-r)S_r}\left((n-r)S_r\tr\left(P_rA^2\right)f_i^2+2(n-r)(r+1)S_rS_{r+1}f_ig_i+(n-r)^2S_r^2g_i^2\right)\,d\mu_\Sigma \nonumber \\
&\leq& -\int_\Sigma\frac{1}{(n-r)S_r}\left((r+1)^2S_{r+1}^2f_i^2+2(n-r)(r+1)S_rS_{r+1}f_ig_i+(n-r)^2S_r^2g_i^2\right)\,d\mu_\Sigma \label{0035} \\
&=& -\int_\Sigma\frac{\left((r+1)S_{r+1}f_i+(n-r)S_rg_i\right)^2}{(n-r)S_r}\,d\mu_\Sigma \leq 0, \label{0036}
\end{eqnarray}
where \eqref{0035} is a consequence of Lemma \ref{009.3}. Since \eqref{0036} holds, we have for all $i \in \{1,..,n\}$,
$$\int_\Sigma\frac{\left((r+1)S_{r+1}f_i+(n-r)S_rg_i\right)^2}{(n-r)S_r}\,d\mu_\Sigma=\int_\Sigma \tr\left(P_rA^2\right)f_i^2+2\left(r+1\right)S_{r+1}f_ig_i+\left(n-r\right)S_rg_i^2\,d\mu_\Sigma=0,$$ hence
\begin{eqnarray}
0 &=& \int_\Sigma\frac{\left((r+1)S_{r+1}f_i+(n-r)S_rg_i\right)^2}{(n-r)S_r}-\left(\tr\left(P_rA^2\right)f_i^2+2\left(r+1\right)S_{r+1}f_ig_i+\left(n-r\right)S_rg_i^2\right)\,d\mu_\Sigma \nonumber \\
&=& \int_\Sigma\frac{\tau_r}{(n-r)S_r}f_i^2\,d\mu_\Sigma. \label{0037}
\end{eqnarray}
Summing up \eqref{0037} for $i \in \{1,...,n-1\}$ we get $\int_\Sigma \frac{\tau_r}{(n-r)S_r}\sum_{i=1}^{n-1}f_i^2\,d\mu_\Sigma=0$, and since $\tau_r \leq 0$ and $H_r>0$, we conclude that $\tau_r\sum_{i=1}^{n-1} f_i^2=0$ on $\Sigma$.

If $n>2$ then the argument used in Theorem \ref{007.1} to prove $\left(\sum_{i=1}^{n-1} f_i^2\right)^{-1}(\{0\})$ has measure zero can be used here and we conclude that $\varphi(\Sigma)$ is totally umbilical. Otherwise, if $n=2$ then, by the holomorphic Hopf differential, it is known that $\Sigma$ is totally umbilical or its umbilic points are isolated. If $\Sigma$ is not totally umbilical we would obtain that $f_1 \equiv 0$. But this implies that $\varphi\left(\Sigma\right)$ is contained in a $3$-dimensional subspace of $\mathbb{R}^4$ and intersecting with $\mathbb{S}^3$ we would conclude that $\varphi\left(\Sigma\right)$ is contained in an equator of $\mathbb{S}^3$, which is a contradiction.
\end{proof}

\section{Symmetric $r$-stability}\label{seis}
Inspired by the results in \cite{elbert2019note}, we provide a notion of symmetric $r$-stability for positive definite capillary $H_{r+1}$-hypersurfaces with constant contact angle $\theta \in (0,\pi)$. In the case $r=0$ or in the case $r>0$ and $M=\mathbb{M}^{n+1}(c)$, the notion of symmetric $r$-stability is equivalent to that given in the Definition \ref{004.3}.
In the case $r>0$ and general ambient space, they do not coincide and we will have, therefore, two different notions of stability to work with. The symmetric stability will equip the theory for $r>0$ in a general ambient space with a bilinear symmetric form, which is the key that allows to mimic part of the classical CMC stability theory.

Let us fix notation before starting the next result. We set
\begin{equation}\label{0075}
Q_r=q_r-\frac{\left\vert\sqrt{P_r}X_r\right\vert^2}{4}-\frac{\dive\left(P_rX_r\right)}{2} =\tr\left(P_r\left(A^2+\overline{R}_\eta\right)\right)-\frac{\left\vert\sqrt{P_r}X_r\right\vert^2}{4}-\frac{\dive\left(P_rX_r\right)}{2},
\end{equation}
where $X_r=-P_r^{-1}\dive P_r$.

\begin{prop}\label{011.1}
Let $\varphi : \Sigma^n \rightarrow \Omega \subseteq M$ be a strongly $r$-stable capillary hypersurface supported on $\partial\Omega$. Then the symmetric bilinear form $\mathcal{I}_{r,\theta}^S : H^1(\Sigma) \times H^1(\Sigma) \rightarrow \mathbb{R}$ defined by $$\mathcal{I}_{r,\theta}^S(f_1,f_2)=\int_\Sigma \left<P_r\nabla f_1,\nabla f_2\right>-Q_rf_1f_2\,d\mu_\Sigma+\int_{\partial\Sigma}\left\vert{P_r\nu}\right\vert\left(\alpha_\theta-\frac{\left<X_r,\nu\right>}{2}\right)f_1f_2\,d\mu_{\partial\Sigma},$$ is positive definite.
\end{prop}

\begin{proof}
Let $f_0 \in C^\infty(\Sigma)$ be a positive function such that $\frac{\partial f_0}{\partial\nu}+\alpha_\theta f_0=0$ on $\partial\Sigma$ and $T_rf_0 \geq 0$ in $\Sigma$. Then denoting $\widetilde{f}=\log f_0 \in C^\infty(\Sigma)$ we have $\nabla\widetilde{f}=\frac{\nabla f_0}{f_0}$ in $\Sigma$ and
\begin{eqnarray}
0 &\leq& \frac{T_rf_0}{f_0}=-\frac{1}{f_0}\dive\left(P_r\nabla f_0\right)+\left<\dive P_r,\frac{\nabla f_0}{f_0}\right>-q_r \nonumber \\
&=& -\dive\left(P_r\frac{\nabla f_0}{f_0}\right)+\left<P_r\nabla f_0,\nabla\left(\frac{1}{f_0}\right)\right>-\left<P_r\left(-P_r^{-1}\dive P_r\right),\nabla\widetilde{f}\right>-q_r \nonumber \\
&=& -\dive\left(P_r\nabla\widetilde{f}\right)-\left<P_r\nabla\widetilde{f},\nabla\widetilde{f}\right>-2\left<P_r\frac{X_r}{2},\nabla\widetilde{f}\right>-q_r \nonumber \\
&=& -\dive\left(P_r\nabla\widetilde{f}\right)-\left\vert{P_r\left(\nabla\widetilde{f}+\frac{X_r}{2}\right)}\right\vert^2+\frac{\vert\sqrt{P_r}X_r\vert^2}{4}-q_r \nonumber \\
&=& -\dive\left(P_r\left(\nabla\widetilde{f}+\frac{X_r}{2}\right)\right)-\left\vert{P_r\left(\nabla\widetilde{f}+\frac{X_r}{2}\right)}\right\vert^2-\left(q_r-\frac{\vert\sqrt{P_r}X_r\vert^2}{4}-\frac{\dive\left(P_rX_r\right)}{2}\right). \label{0053}
\end{eqnarray}
Denoting $Y_r=\nabla\widetilde{f}+\frac{X_r}{2}$ we obtain that $-\dive\left(P_r Y_r\right)-\vert\sqrt{P_r}Y_r\vert^2-Q_r \geq 0$. Thus, if $f \in C^\infty(\Sigma)$ then \eqref{0053} gives
\begin{eqnarray}
0 &\leq& \int_\Sigma -f^2\dive\left(P_rY_r\right)-f^2\vert\sqrt{P_r}Y_r\vert^2-Q_rf^2\,d\mu_\Sigma \nonumber \\
&=& \int_\Sigma \left<P_rY_r,\nabla\left(f^2\right)\right>-f^2\vert\sqrt{P_r}Y_r\vert^2-Q_rf^2\,d\mu_\Sigma-\int_\Sigma \dive\left(f^2P_rY_r\right)\,d\mu_\Sigma \nonumber \\
&=& \int_\Sigma 2f\left<\sqrt{P_r}\nabla f,\sqrt{P_r}Y_r\right>-f^2\vert\sqrt{P_r}Y_r\vert^2-Q_rf^2\,d\mu_\Sigma-\int_{\partial\Sigma} f^2\left<P_rY_r,\nu\right>\,d\mu_{\partial\Sigma}. \label{0054}
\end{eqnarray}
Since the Cauchy-Schwarz and geometric-arithmetic inequality gives $$2f\left<\sqrt{P_r}\nabla f,\sqrt{P_r}Y_r\right> \leq 2\vert{f}\vert\vert\sqrt{P_r}\nabla f\vert\vert\sqrt{P_r}Y_r\vert \leq \vert\sqrt{P_r}\nabla f\vert^2+f^2\vert\sqrt{P_r}Y_r\vert^2,$$ \eqref{0054} implies that
\begin{equation}\label{0055}
\int_\Sigma \vert\sqrt{P_r}\nabla f\vert^2-Q_rf^2\,d\mu_\Sigma-\int_{\partial\Sigma} \left\vert{P_r\nu}\right\vert\left<\nabla\widetilde{f}+\frac{X_r}{2},\nu\right>f^2\,d\mu_{\partial\Sigma} \geq 0
\end{equation}
for all $f \in C^\infty(\Sigma)$. Also since $\left<\nabla\widetilde{f},\nu\right>=\frac{1}{f_0}\frac{\partial f_0}{\partial\nu}=-\alpha_\theta$ on $\partial\Sigma$, we conclude from \eqref{0052} and \eqref{0055} that $$\mathcal{I}_{r,\theta}^S(f,f)=\int_\Sigma \left<P_r\nabla f,\nabla f\right>-Q_rf^2\,d\mu_\Sigma+\int_{\partial\Sigma}\left\vert{P_r\nu}\right\vert\left(\alpha_\theta-\frac{\left<X_r,\nu\right>}{2}\right)f^2\,d\mu_{\partial\Sigma} \geq 0,$$ proving the claim.
\end{proof}

Inspired by the above result, we can revisit some of the definitions we gave and obtain their ``symmetrized" versions.
\begin{defn}\label{011.2}
For a positive definite $H_{r+1}$-hypersurface $\varphi : \Sigma^n \rightarrow M$ we define the \textbf{symmetric r-index form} of $\varphi$ by
\begin{multline}\label{0052}
\mathcal{I}_{r,\theta}^S(f_1,f_2)=\int_\Sigma \left<P_r\nabla f_1,\nabla f_2\right>-Q_rf_1f_2\,d\mu_\Sigma+\\+\int_{\partial\Sigma}\left\vert{P_r\nu}\right\vert\left(\csc\theta\left(\emph{II}_{\partial\Omega}\right)_{\overline\eta}(\overline\nu,\overline\nu)-\cot\theta\left(\emph{II}_\Sigma\right)_\eta(\nu,\nu)-\frac{\left<X_r,\nu\right>}{2}\right)f_1f_2\,d\mu_{\partial\Sigma}.
\end{multline}
\end{defn}

\begin{defn}\label{011.3}
A positive definite capillary $H_{r+1}$-hypersurface $\varphi : \Sigma^n \rightarrow \Omega \subseteq M$ supported on $\partial\Omega$ is \textbf{symmetric $r$-stable} if $\mathcal{I}_{r,\theta}^S(f,f) \geq 0$ for all $f \in \mathcal{F}$. If the inequality holds for all $f \in H^1(\Sigma)$ then the hypersurface is called \textbf{strongly symmetric $r$-stable}.
\end{defn}

We notice that when $r=0$ or $M=\mathbb{M}^{n+1}(c)$, $X_r=0$ in \eqref{0075} and, in this case, $\mathcal{I}_{r,\theta}^S$ given by \eqref{0052} coincides with the index formula given in Definition \ref{005.1} and $r$-stability and symmetric $r$-stability coincide. In the general case, from Proposition \ref{011.1} we conclude that
\begin{cor}\label{011.4}
Let $\varphi : \Sigma \rightarrow M$ be a positive definite capillary $H_{r+1}$-hypersurface supported on $\partial\Omega$. If $\Sigma$ is $r$-stable then it is symmetric $r$-stable. 
\end{cor}

As well as the bilinear form given in \eqref{0010}, the bilinear form \eqref{0052} is also associated to a differential operator.
\begin{defn}\label{011.5}
The \textbf{symmetric $r$-stability operator} for a positive definite $H_{r+1}$-hypersurface $\varphi : \Sigma^n \rightarrow M$ is given by
\begin{equation}\label{0077}
T_r^S=-\dive\left(P_r\nabla\cdot\right)-Q_r.
\end{equation}
\end{defn}

\section{Stability of cylinders in $M \times \mathbb{R}$}\label{sete}
We start this section organizing the most relevant definitions of the symmetric stability theory for $H_{r+1}$-hypersurfaces, based on the approach of the present paper (free boundary) and of \cite{elbert2019note} (fixed boundary or empty boundary).

A capillary positive definite $H_{r+1}$-hypersurface $\varphi : \Sigma^n \rightarrow \Omega \subseteq M$ supported on $\partial\Omega$ is
\begin{enumerate}[(i)]
\item symmetric $r$-stable if $\mathcal{I}_{r,\theta}^S(f,f) \geq 0$ for all $f \in \left\{f \in H^1(\Sigma)\,|\,\int_\Sigma f\,d\mu_\Sigma=0\right\}$;
\item strongly symmetric $r$-stable if $\mathcal{I}_{r,\theta}^S(f,f) \geq 0$ for all $f \in H^1(\Sigma)$.
\end{enumerate}

For a positive definite $H_{r+1}$-hypersurface with closed or fixed boundary $\varphi : \Sigma^n \rightarrow M$ the index form that replaces \eqref{0052} is $I_r^S : H_0^1(\Sigma) \times H_0^1(\Sigma) \rightarrow \mathbb{R}$, where
\begin{equation}\label{0076}
I_r^S(f_1,f_2)=\int_\Sigma \left<P_r\nabla f_1,\nabla f_2\right>-Q_rf_1f_2\,d\mu_\Sigma,
\end{equation}
and we say that $\varphi : \Sigma^n \rightarrow M$ is
\begin{enumerate}[(i)]
\item symmetric $r$-stable if $I_r^S(f,f) \geq 0$ for all $f \in \left\{f \in H_0^1(\Sigma)\,|\,\int_\Sigma f\,d\mu_\Sigma=0\right\}$;
\item strongly symmetric $r$-stable if $I_r^S(f,f) \geq 0$ for all $f \in H_0^1(\Sigma)$,
\end{enumerate}
where $H_0^1(\Sigma)$ is the closure of $C_0^\infty(\Sigma)$ with respect to the norm $\left\Vert\cdot\right\Vert_{H^1(\Sigma)}$.
 
Associated to the index form, we set $(T_r^S,B)$ to denote the symmetric $r$-stability operator \eqref{0077} with the boundary condition $B(f)=0$ on $\partial\Sigma$, where
\begin{equation}\label{0078}
B(f)=\begin{dcases*}
\frac{\partial f}{\partial\nu}+\alpha_\theta^S f,& for capillary (Robin condition)\\
f,& for closed or fixed boundary (Dirichlet condition)
\end{dcases*},
\end{equation}
where $\alpha_\theta^S=\alpha_\theta-\frac{\left<X_r,\nu\right>}{2} \in C^\infty(\partial\Sigma)$. For each boundary condition $B(f)=0$, we can consider an eigenvalue problem which is related to the corresponding stability problem.

Now we particularize to the eigenvalue problem with Dirichlet boundary condition, namely, 
\begin{equation}\label{0056}
\begin{dcases*}
T^S_rf=-\dive\left(P_r\nabla f\right)-Q_rf=\lambda f,& in $\Sigma$ \\
f=0,& on $\partial\Sigma$
\end{dcases*},
\end{equation}
with $f \in H_0^1(\Sigma) \backslash \{0\}$. Let $\lambda_1<\lambda_2$ be the first and the second eigenvalues of this problem. For $l \in \{1,2\}$, let $E_{\lambda_l}$ be the eigenspace of $H_0^1(\Sigma)$ associated to $\lambda_l$ and $$E_{\lambda_l}^\perp=\left\{f \in H_0^1(\Sigma)\,|\,\left<f,g\right>_{L^2(\Sigma)}=0 \text{ for any }g \in E_{\lambda_l}\right\}.$$ Then
\begin{equation}\label{0057}
\lambda_1=I_r^S(f_1,f_1)=\min\left\{I_r^S(f,f)\,\left|\right.\,f \in H_0^1(\Sigma)\text{ and }\int_\Sigma f^2\,d\mu_\Sigma=1\right\}
\end{equation}
and
\begin{equation}\label{0058}
\lambda_2=I_r^S(f_2,f_2)=\min\left\{I_r^S(f,f)\,\left|\right.\,f \in H_0^1(\Sigma) \cap E_{\lambda_1}^\perp \text{ and }\int_\Sigma f^2\,d\mu_\Sigma=1\right\},
\end{equation}
where $f_1$ and $f_2$ are elements of an orthonormal basis $\{f_l\}_{l \in \mathbb{N}}$ for $L^2(\Sigma)$ composed by eigenfunctions of \eqref{0056}.

The next lemma, which is a generalization of a result proven by M. Koiso in \cite[Theorem 1.3]{koiso2002deformation} and was used by R. Souam in \cite[Theorem 3.1]{souam2021stable}, gives criteria for the symmetric $r$-stability for closed hypersurfaces. The proof is essentially the same of that in \cite{koiso2002deformation} and we include it on Appendix C for completeness.

\begin{lema}\label{012.1}
Let $\varphi : \Sigma \rightarrow M$ be a closed $H_{r+1}$-hypersurface. The following hold:
\begin{enumerate}[(i)]
\item $\lambda_1 \geq 0$ if and only if $\Sigma$ is strongly symmetric $r$-stable.
\item If $\lambda_1<0<\lambda_2$, then there exists a unique function $f \in H_0^1(\Sigma)$ such that $T_r^Sf=-1$, and $\Sigma$ is symmetric $r$-stable if and only if $\int_\Sigma f\,d\mu_\Sigma \geq 0$.
\item If $\lambda_1<0=\lambda_2$ and there exists $f \in E_{\lambda_2}$ satisfying $\int_\Sigma f\,d\mu_\Sigma \neq 0$, then $\Sigma$ is symmetric $r$-unstable.
\item If $\lambda_1<0=\lambda_2$ and $\int_\Sigma f\,d\mu_\Sigma=0$ for any $f \in E_{\lambda_2}$ then there exists a unique $f \in H_0^1(\Sigma) \cap E_{\lambda_2}^\perp$ such that $T_r^Sf=-1$ and $\Sigma$ is symmetric $r$-stable if and only if $\int_\Sigma f\,d\mu_\Sigma \geq 0$.
\item if $\lambda_2<0$, then $\Sigma$ is symmetric $r$-unstable.
\end{enumerate}
\end{lema}

 The following result is inspired on \cite[Theorem 3.3]{souam2021stable}.

\begin{teo}\label{012.2}
Let $\varphi_0 : \Sigma_0 \rightarrow M$ be a closed oriented and positive definite $H_{r+1}^{(0)}$-hypersurface and let $l>0$. The map $\widetilde\varphi:=\varphi_0 \times \id_{[0,l]} : \Sigma:=\Sigma_0 \times [0,l] \rightarrow M \times \mathbb{R}$ is a positive definite free boundary hypersurface with $(r+1)$-th order mean curvature equal to $\frac{n-r-1}{n}H_{r+1}^{(0)}$. If $\lambda_1^{(0)}$ is the first eigenvalue of $T_r^S$ on $\Sigma_0$. Then we have

\begin{enumerate}[(i)]
\item If $\Sigma_0$ is symmetric $r$-unstable then $\Sigma$ is symmetric $r$-unstable.
\item Suppose that $\Sigma_0$ is symmetric $r$-stable.
\begin{description}
\item[a)] Assume, in addition, that $\Sigma$ is symmetric $r$-stable. Then $\lambda_1^{(0)}+S_r^{(0)}\frac{\pi^2}{l^2}\geq 0$
\item[b)] Assume, in addition, that $S_r^{(0)}$ constant and that $\lambda_1^{(0)}+S_r^{(0)}\frac{\pi^2}{l^2} \geq 0$. Then $\Sigma$ is symmetric $r$-stable.
\end{description}
\end{enumerate}
\end{teo}

\begin{proof}
 Denote with a superscript $^{(0)}$ the quantities related to $\Sigma_0$ and by $t$ the global coordinate of $\mathbb{R}$. If $\kappa_1^{(0)},...,\kappa_{n-1}^{(0)}$ denotes the principal curvatures of $\Sigma_0$ associated with the eigenvectors $\{e_1^{(0)},...,e_{n-1}^{(0)}\}$ and $\pi_M$ is the projection of $M \times \mathbb{R}$ onto $M$, then the principal curvatures $\kappa_1,...,\kappa_n$ of $\Sigma$ are equal to $$\kappa_i=\begin{dcases*}\kappa_i^{(0)} \circ \pi_M,& $i \neq n$ \\ 0,& $i=n$\end{dcases*}$$ and the associated eigenvectors are $\left\{e_1,..,e_{n-1},e_n=\left.\frac{\partial}{\partial t}\right\vert_\Sigma\right\}$, where $e_i \in \Gamma(T\Sigma)$ is the horizontal vector field such that $e_i^{(0)}=d\pi_Me_i$. The eigenvalues $S_r(A_i)$ of the Newton transformation $P_r$ of $\Sigma$ are equal to $S_r^{(0)} \circ \pi_M$ if $i=n$ and
\begin{eqnarray*}
S_r(A_i) &=& \sum_{1 \leq i_1<...<i_r \leq n; i_1,...,i_r \neq i} \kappa_{i_1}\cdot ... \cdot \kappa_{i_r} \\
&=& \kappa_n\sum_{1 \leq i_1<...<i_{r-1} \leq n-1; i_1,...,i_{r-1} \neq i} \kappa_{i_1}\cdot ... \cdot \kappa_{i_{r-1}}+\sum_{1 \leq i_1<...<i_r \leq n-1; i_1,...,i_r \neq i}\kappa_{i_1}\cdot ... \cdot \kappa_{i_r} \\
&=& S_r^{(0)}(A_i^{(0)}) \circ \pi_M, \quad i \neq n,
\end{eqnarray*}
proving that $\Sigma$ is positive definite provided $\Sigma_0$ is also positive definite. Also, its $(r+1)$-th order mean curvature is equal to
\begin{eqnarray*}
S_{r+1} &=& \sum_{1 \leq i_1<...<i_{r+1} \leq n} \kappa_{i_1}\cdot ... \cdot \kappa_{i_r} \\
&=& \kappa_n\sum_{1 \leq i_1<...<i_r \leq n-1}\kappa_{i_1}\cdot ... \cdot \kappa_{i_r}+\sum_{1 \leq i_1<...<i_{r+1} \leq n-1} \kappa_{i_1}\cdot ... \cdot \kappa_{i_r}=S_{r+1}^{(0)},
\end{eqnarray*}
proving that $$H_{r+1}=\binom{n}{r+1}^{-1}S_{r+1}=\binom{n}{r+1}^{-1}S_{r+1}^{(0)}=\dfrac{\binom{n-1}{r+1}}{\binom{n}{r+1}}H_{r+1}^{(0)}=\frac{n-r-1}{n}H_{r+1}^{(0)}.$$

Let $\{u_1,..,u_{n-1}, u_n=\frac{\partial}{\partial t}\}$ be a geodesic frame of $\Sigma$ centered at some point $p=(p_0,t) \in \Sigma$. Then we have, $$\dive_\Sigma P_r=\sum_{i=1}^n \nabla_{u_i}(P_r(u_i))=\sum_{i=1}^{n-1}\nabla_{u_i}(P_r(u_i)),$$ proving that $\dive_\Sigma P_r$ and, {\it a fortiori}, $X_r$ are horizontal vector fields whose projections onto $T\Sigma_0$ are equal to $\dive_{\Sigma_0}P_r^{(0)}$ and $X_r^{(0)}$, respectively. Hence, $\left\vert\sqrt{P_r}X_r\right\vert=\left\vert\sqrt{P_r^{(0)}}X_r^{(0)}\right\vert_0$, and that $\dive_\Sigma\left(P_rX_r\right)$ is a horizontal vector field whose projection onto $T\Sigma_0$ is equal to $\dive_{\Sigma_0}\left(P_r^{(0)}X_r^{(0)}\right)$. Also, since $\nu=\begin{dcases*}-e_n,& $t=0$ \\ e_n,& $t=l$\end{dcases*}$ on $\partial\Sigma=\Sigma_0 \times \{0,l\}$, we have that $\left<X_r,\nu\right>=0$ on $\partial\Sigma$. Also since, $\tr_\Sigma\left(P_rA^2\right)=\tr_{\Sigma_0}\left(\left(P_rA^2\right)^{(0)}\right)\circ\pi_M$ on $\Sigma$ and the unit normal vector field $\eta \in \Gamma(N\Sigma)$ is the horizontal vector field of $M \times \mathbb{R}$ whose projection onto $TM$ is the unit normal vector field $\eta^{(0)}$ of $\Sigma_0$ and $\overline{R}(\eta,\partial_t)\eta=0$, we have that $\tr_\Sigma\left(P_r\overline{R}_\eta\right)=\tr_{\Sigma_0}\left(P_r^{(0)}\overline{R}_{\eta^{(0)}}^{(0)}\right) \circ \pi_M$, where $\overline{R}^{(0)}$ is the curvature tensor of $M$. Hence, $Q_r=Q_r^{(0)} \circ \pi_M$ on $\Sigma$.

Suppose that $\Sigma_0$ is a symmetric $r$-unstable hypersurface. Let $f_0 \in \mathcal{F}_{\Sigma_0}$ be such that $I_r^S(f_0,f_0)<0$ then, for $\tilde{f}=f_0 \circ \pi_M \in \mathcal{F}_\Sigma$, we obtain that $\mathcal{I}_r^S(\tilde{f},\tilde{f})=I_r^S(f_0,f_0)<0$. Thus $\Sigma$ is symmetric $r$-unstable.

Now suppose that $\Sigma_0$ is symmetric $r$-stable and assume that $\lambda_1^{(0)}+S_r^{(0)}\frac{\pi^2}{l^2}<0$. Then if $f_0 \in H^1(\Sigma_0)$ is an eigenfunction of $T_r^S$ associated to $\lambda_1^{(0)}$, we have for $\tilde{f}(p)=f_0(p_0)\cdot\cos\left(\frac{\pi t}{l}\right) \in H^1(\Sigma)$ $$\int_\Sigma \tilde{f}\,d\mu_\Sigma=\int_0^l\left(\int_{\Sigma_0}f_0\,d\mu_{\Sigma_0}\right)\cos\frac{\pi t}{l}\,dt=0,$$ showing that $\tilde{f} \in \mathcal{F}_\Sigma$. Moreover,
\begin{eqnarray*}
T_r^S\tilde{f} &=& -\dive_\Sigma\left(P_r\nabla\tilde{f}\right)-Q_r\tilde{f} \\
&=& -\dive_\Sigma\left(\cos\frac{\pi t}{l}P_r\nabla f_0-\frac{\pi}{l}\sin\frac{\pi t}{l}P_r\frac{\partial}{\partial t}\right)-Q_rf_0\cos\frac{\pi t}{l} \\
&=& \left(-\dive_{\Sigma_0}\left(P_r^{(0)}\nabla^{(0)}f_0\right)+Q_rf_0\right)\cos\frac{\pi t}{l}+S_r^{(0)}\frac{\pi^2}{l^2}f_0\cos\frac{\pi t}{l} \\
&=& \left(\tilde\lambda_1+S_r^{(0)}\frac{\pi^2}{l^2}\right)\tilde{f}.
\end{eqnarray*}
Thus, $\mathcal{I}_r^S(\tilde{f},\tilde{f})=\int_\Sigma \left(\tilde\lambda_1+S_r^{(0)}\frac{\pi^2}{l^2}\right)\tilde{f}^2\,d\mu_\Sigma<0$, proving that $\Sigma$ is symmetric $r$-unstable.

Now suppose that $S_r^{(0)}$ is constant and that $\tilde\lambda_1+S_r^{(0)}\frac{\pi^2}{l^2} \geq 0$. Consider the immersion $\widehat\varphi :=\varphi \times \id_{\mathbb{S}^1\left(\frac{l}{\pi}\right)} : \widehat\Sigma:=\Sigma_0 \times \mathbb{S}^1\left(\frac{l}{\pi}\right) \rightarrow M \times \mathbb{S}^1\left(\frac{l}{\pi}\right)$ and denote with a hat $\widehat\cdot$ the quantities related to $\widehat\varphi$. The curvatures of $\widehat\Sigma$ are the equal to those of $\Sigma$ and if $\widehat\eta$, $\widehat{A}$, $\widehat{P_r}$ are the unit normal vector field, second fundamental form and the $r$-th order Newton transformation of $\widehat\varphi$ and $\widehat{R}$ is the Riemann curvature tensor of $M \times \mathbb{S}^1\left(\frac{l}{\pi}\right)$ then $\dive_{\widehat\Sigma}\left(\widehat{P}_r\widehat\nabla\cdot\right)=\dive_\Sigma\left(P_r\nabla\cdot\right)$ and $\widehat{Q}_r=Q_r$ in $\widehat\Sigma$. The eigenvalues and eigenfunctions of the eigenvalue problem $f^{\prime\prime}+\frac{\mu}{S_r^{(0)}}f=0$ on $\mathbb{S}^1\left(\frac{l}{\pi}\right)$ are given by $\mu_m=S_r^{(0)}\frac{m^2\pi^2}{l^2}$ for $m \in \mathbb{N} \cup \{0\}$ and $f_m(t)=\cos\left(\sqrt{\frac{\mu_m}{S_r^{(0)}}}t\right)$. Let $\lambda_1^{(0)}<\lambda_2^{(0)} \leq ... \leq \lambda_k^{(0)} \nearrow +\infty$ be the sequence of eigenvalues of $T_r^{S,(0)}$ on $\Sigma_0$. Since $\widehat{Q}_r=Q_r$ depends only on the coordinates of $\Sigma_0$, the eigenvalues of $\widehat{T}_r^S$ on $\widehat\Sigma$ are given by $\lambda_k^{(0)}+S_r^{(0)}\frac{m^2\pi^2}{l^2}$, with $(k,m) \in \mathbb{N} \times \left(\mathbb{N} \cup \{0\}\right)$. In particular, the first two eigenvalues of \eqref{0056} in $\widehat\Sigma$ are equal to
\begin{eqnarray*}
\widehat\lambda_1 &=& \lambda_1^{(0)} \\
\widehat\lambda_2 &=& \min\left\{\lambda_1^{(0)}+S_r^{(0)}\frac{\pi^2}{l^2},\lambda_2^{(0)}\right\}.
\end{eqnarray*}
If $\lambda_1^{(0)} \geq 0$ then $\widehat\lambda_1 \geq 0$, which implies that $\widehat\Sigma$ is strongly symmetric $r$-stable. So we assume $\lambda_1^{(0)}<0$. Since $\Sigma_0$ is symmetric $r$-stable, it follows from Lemma \ref{012.1} that $\lambda_2^{(0)} \geq 0$. Thus, $\widehat\lambda_2 \geq 0$. Consider the solution $f \in C^\infty(\Sigma_0)$ to the equation $T_r^Sf=-1$ described on the items (ii) and (iv) of Lemma \ref{012.1}. Then the function $\widetilde{f}(p_0,t)=f(p_0) \in C^\infty(\widehat\Sigma)$ satisfies $\widehat{T}_r^S\widehat{f}=T_r^Sf=-1$ and $$\int_{\widehat\Sigma}\widehat{f}\,d\mu_{\widehat\Sigma}=2l\int_{\Sigma_0}f\,d\mu_{\Sigma_0} \geq 0.$$ Hence, it follows from the items (ii) and (iv) of Lemma \ref{012.1} that $\widehat\Sigma$ is symmetric $r$-stable.

Now we will show the symmetrical $r$-stability of $\widehat\Sigma$ implies the symmetrical $r$-stability of $\Sigma$. In fact, let $f \in \mathcal{F}_\Sigma$ and extend $f$ to a function in $H^1(\Sigma_0 \times [0,2l])$ denoting $f(p_0,t)=f(p_0,2l-t)$ for $(p_0,t) \in \Sigma_0 \times [l,2l]$. Since $f(p_0,2l)=f(p_0,0)$ for all $p_0 \in \Sigma_0$ we obtain a function $\widehat{f} \in H^1(\widehat\Sigma)$ such that $\int_{\widehat\Sigma}\widehat{f}\,d\mu_{\widehat\Sigma}=0$ and $\mathcal{I}_r^S(f,f)=\frac{1}{2}\widehat{\mathcal{I}}_r^S(\widehat{f},\widehat{f}) \geq 0$, proving that $\Sigma$ is symmetric $r$-stable.
\end{proof}

As a consequence of Theorem \ref{012.2} we obtain a characterization of symmetric $r$-stable tubes $\partial B_R^{\mathbb{M}^n(c)} \times [0,l] \subseteq \mathbb{M}^n(c) \times [0,l]$ with radius $R>0$ and height $l>0$. Here we will consider the warped product model $\mathbb{M}^n(c)=[0,R_c) \times_{\sn_c} \mathbb{S}^{n-1}$, where $R_c=+\infty$ if $c<0$ and $R_c=\frac{\pi}{\sqrt{c}}$ if $c>0$. Since geodesic spheres with radius $R \in (0,R_c)$ are totally umbilical hypersurfaces with constant curvature equal to $\frac{\cn_c(R)}{\sn_c(R)}$, we have that $S_r^{(0)}=\binom{n-r-1}{n-1}\left(\frac{\cn_c(R)}{\sn_c(R)}\right)^r$ and $P_r^{(0)}=\frac{n-r-1}{n-1}\binom{n-r-1}{n-1}\left(\frac{\cn_c(R)}{\sn_c(R)}\right)^rI$. Hence,
\begin{eqnarray}
T_r^{(0)} &=& -\frac{n-r-1}{n-1}\binom{n-r-1}{n-1}\left(\frac{\cn_c(R)}{\sn_c(R)}\right)^r\left(\Delta+(n-1)\left(\frac{\cn_c^2(R)}{\sn_c^2(R)}+c\right)\right) \nonumber\\
&=& \frac{n-r-1}{n-1}\binom{n-1}{r}s_c^r(R)T_0^{(0)}. \label{0062}
\end{eqnarray}
Now since the first closed eigenvalue of $-\Delta$ on $\mathbb{S}^{n-1}$ is equal to $0$ (see \cite[p. 34]{chavel1984eigenvalues}), the first eigenvalue of $-\Delta$ on $\partial B_R=\left(\mathbb{S}^{n-1},\frac{\cn_c^2(R)}{\sn_c^2(R)}\,g_{\mathbb{S}^{n-1}}\right)$ is given by $\lambda_1(-\Delta,\partial B_R)=\frac{\lambda_1(-\Delta,\mathbb{S}^{n-1})}{f_c(R)^2}=0$. Thus, the first eigenvalue $\lambda_1^{(0)}$ of \eqref{0062} is equal to
\begin{eqnarray}
\lambda_1^{(0)} &=& -(n-r-1)\binom{n-1}{r}\cdot \left(\frac{\cn_c(R)}{\sn_c(R)}\right)^r \cdot \left(\frac{\cn_c^2(R)}{\sn_c^2(R)}+c\right) \nonumber \\
&=& -(n-r-1)\binom{n-1}{r} \cdot \frac{1}{\sn_c^2(R)} \cdot \left(\frac{\cn_c(R)}{\sn_c(R)}\right)^r \label{0063}
,\end{eqnarray}
where the last equation is a consequence of the identity $\cn_c^2(\rho)+c\sn_c^2(\rho)=1$ for all $\rho \in [0,R_c)$. Also, since geodesic spheres on space forms are $r$-stable \cite[Proposition 5.1]{marques1997stability}, we obtain a generalization of \cite[Corollary 3.4]{souam2021stable}.

\noindent\begin{cor}\label{012.3}

\ 

\begin{enumerate}[(i)]
\item A tube of radius $R>0$ and height $l>0$ in $\mathbb{H}^n(c) \times [0,l]$ is symmetric $r$-stable if and only if $\frac{\pi\sinh\left(R\sqrt{-c}\right)}{\sqrt{-c}} \geq l\sqrt{n-r-1}$.
\item A tube of radius $R>0$ and height $l>0$ in $\mathbb{R}^n \times [0,l]$ is $r$-stable if and only if $\pi R \geq l\sqrt{n-r-1}$.
\item A tube of radius $R>0$ and height $l>0$ in $\mathbb{S}^n(c) \times [0,l]$ is symmetric $r$-stable if and only if $\frac{\pi\sin\left(R\sqrt{c}\right)}{\sqrt{c}} \geq l\sqrt{n-r-1}$.
\end{enumerate}
\end{cor}

\begin{proof}
Since $\partial B_R$ is $r$-stable, it follows from Theorem \ref{012.2} that $\partial B_R \times [0,l]$ is symmetric $r$-stable if and only if $\lambda_1^{(0)}+S_r^{(0)}\frac{\pi^2}{l^2} \geq 0$. But since $S_r^{(0)}=\binom{n-1}{r}\left(\frac{\cn_c(R)}{\sn_c(R)}\right)^r$, if follows from \eqref{0063} that a tube is symmetrically stable if and only if
\begin{eqnarray*}
\binom{n-1}{r} \cdot \left(\frac{\cn_c(R)}{\sn_c(R)}\right)^r \cdot \left(-\frac{n-r-1}{\sn_c^2(R)}+\frac{\pi^2}{l^2}\right) \geq 0 &\iff& \frac{\pi^2}{l^2} \geq \frac{n-r-1}{\sn_c^2(R)} \\
&\iff& \pi\sn_c(R) \geq l\sqrt{n-r-1},
\end{eqnarray*}
proving the result.
\end{proof}

\appendix
\section{Proof of Theorem \ref{004.2}}
For completeness we will give a proof of Proposition \ref{004.2}. The computations made here are similar of when $r=0$, done by Ros and Souam in \cite[Section 4]{ros1997stability}. Let $\Phi : \Sigma^n \times (-\varepsilon,\varepsilon) \rightarrow \Omega \subseteq M$ be an admissible volume-preserving variation of a capillary $H_{r+1}$-hypersurface $\varphi : \Sigma \rightarrow M$ supported on $\partial\Omega$. The derivative of \eqref{0005} is equal to
\begin{equation}\label{0064}
\left.\frac{\partial}{\partial t}\mathcal{F}_r[\Sigma_t]\right\vert_{t=0}=-\int_\Sigma \left.\frac{\partial}{\partial t}\left(S_{r+1}(t)\left<\xi_t,\eta_t\right>\,d\mu_{\Sigma_t}\right)\right\vert_{t=0}+\int_{\partial\Sigma} \left.\frac{\partial}{\partial t}\left(\left<\xi_t,(P_r\nu-\vert{P_r\nu}\vert\cos\theta\,\overline\nu)_t\right>\,d\mu_{\partial\Sigma_t}\right)\right\vert_{t=0}.
\end{equation}

Set
\begin{eqnarray}
I_1 &=& -\int_\Sigma \left.\frac{\partial}{\partial t}\left(S_{r+1}(t)\left<\xi_t,\eta_t\right>\,d\mu_{\Sigma_t}\right)\right\vert_{t=0} \label{0065} \\
I_2 &=& \int_{\partial\Sigma} \left.\frac{\partial}{\partial t}\left(\left<\xi_t,\left(P_r\nu-\vert{P_r\nu}\vert\cos\theta\,\overline\nu\right)_t\right>\,d\mu_{\partial\Sigma_t}\right)\right\vert_{t=0} \label{0066}
\end{eqnarray}
Since $S_{r+1}(0)$ is constant and $\Phi$ is volume-preserving, applying \eqref{0004} to \eqref{0065} yields to
\begin{eqnarray}
I_1 &=& -\int_\Sigma S_{r+1}^\prime(0)\left<\xi,\eta\right>\,d\mu_\Sigma \nonumber \\
&=& -\int_\Sigma f\left(L_rf+\left(S_1S_{r+1}-(r+2)S_{r+2}\right)f+\tr\left(P_r\overline{R}_\eta\right)f\right)\,d\mu_\Sigma, \label{0067}
\end{eqnarray}
where $f=\left<\xi,\eta\right> \in C^\infty(\Sigma)$ is the support function of $\Phi$ at $t=0$. The Lemma \ref{002.4} with \eqref{0002} and the fact that $\xi\vert_{\partial\Sigma} \in \Gamma\left(T\partial\Omega\vert_{\partial\Sigma}\right)$ imply that $$\left<\xi,P_r\nu-\vert{P_r\nu}\vert\cos\theta\,\overline\nu\right>=\vert{P_r\nu}\vert\left<\xi,\nu-\cos\theta\,\overline\nu\right>=\vert{P_r\nu}\vert\left<\xi,\sin\theta\,\overline\eta\right>=0$$ along $\partial\Sigma$. Thus, the second term of \eqref{0064} is equal to 
\begin{equation}\label{0068}
I_2=\int_{\partial\Sigma} \left(\left<\overline\nabla_\xi\xi,P_r\nu-\vert{P_r\nu}\vert\cos\theta\,\overline\nu\right>+\left<\xi,\overline\nabla_\xi\left(P_r\nu-\vert{P_r\nu}\vert\cos\theta\,\overline\nu\right)\right>\right)\,d\mu_{\partial\Sigma}.
\end{equation}

Since
\begin{eqnarray}
\left<\overline\nabla_\xi\xi,P_r\nu-\vert{P_r\nu}\vert\cos\theta\,\overline\nu\right> &=& \vert{P_r\nu}\vert\left<\overline\nabla_\xi\xi,\nu-\cos\theta\,\overline\nu\right> \label{0069} \\
\left<\xi,\overline\nabla_\xi\left(P_r\nu-\vert{P_r\nu}\vert\cos\theta\,\overline\nu\right)\right> &=& \left<\xi,\overline\nabla_{P_r\nu}\xi+\left[\xi,P_r\nu\right]-\cos\theta\,\overline\nabla_\xi\left(\vert{P_r\nu}\vert\,\overline\nu\right)\right> \nonumber \\
&=& \left<\xi,\vert{P_r\nu}\vert\overline\nabla_\nu\xi+\left[\xi,\vert{P_r\nu}\vert\,\nu\right]-\cos\theta\left(\xi\vert{P_r\nu}\vert\,\overline\nu+\vert{P_r\nu}\vert\,\overline\nabla_\xi\overline\nu\right)\right> \nonumber \\
&=& \left<\xi,\vert{P_r\nu}\vert\left(\overline\nabla_\nu\xi+\left[\xi,\nu\right]-\cos\theta\,\overline\nabla_\xi\overline\nu\right)+\xi\vert{P_r\nu}\vert\left(\nu-\cos\theta\,\overline\nu\right)\right> \nonumber \\
&=& \vert{P_r\nu}\vert\left<\xi,\overline\nabla_\xi\nu-\cos\theta\,\overline\nabla_\xi\overline\nu\right>+\xi\vert{P_r\nu}\vert\left<\xi,\sin\theta\,\overline\eta\right> \nonumber \\
&=& \vert{P_r\nu}\vert\left<\xi,\overline\nabla_\xi\nu-\cos\theta\,\overline\nabla_\xi\overline\nu\right>. \label{0070}
\end{eqnarray}
the equations \eqref{0068}, \eqref{0069} and \eqref{0070} give
\begin{equation}\label{0071}
I_2=\int_{\partial\Sigma}\vert{P_r\nu}\vert\left(\left<\overline\nabla_\xi\xi,\nu-\cos\theta\,\overline\nu\right>+\left<\xi,\overline\nabla_\xi\nu-\cos\theta\,\overline\nabla_\xi\overline\nu\right>\right)\,d\mu_{\partial\Sigma}.
\end{equation}

Notice that in the boundary $\partial\Sigma$, the variational field can be written as $\xi=\widetilde\xi+\left<\xi,\nu\right>\nu+f\eta$, where $\widetilde\xi \in \Gamma\left(T\partial\Sigma\right)$ is the projection of $\xi$ onto $T\partial\Sigma$. From the fact that $\xi\vert{\partial\Sigma} \in \Gamma\left(T\partial\Omega\vert_{\partial\Sigma}\right)$ and the relation between $\overline\nu$ and $\{\nu,\eta\}$, we have $$0=\left<\xi,\overline\eta\right>=\left<\xi,\sin\theta\,\nu+\cos\theta\,\eta\right> \implies \left<\xi,\nu\right>=-f\cot\theta.$$ Thus,
\begin{eqnarray}
\xi &=& \widetilde\xi-f\cot\theta\,\nu+f\eta=\widetilde\xi-\frac{f}{\sin\theta}\left(\cos\theta\,\nu-\sin\theta\,\eta\right) \nonumber \\
&=& \widetilde\xi-f\csc\theta\,\overline\nu \label{0074}
.\end{eqnarray}
From \eqref{0002} and \eqref{0074},
\begin{eqnarray}
\left<\overline\nabla_\xi\xi,\nu-\cos\theta\,\overline\nu\right> &=& \sin\theta\left<\overline\nabla_\xi\xi,\overline\nu\right>=\sin\theta\left(\emph{II}_{\partial\Omega}\right)_{\overline\eta}(\xi,\xi) \nonumber \\
&=& \sin\theta\left(\emph{II}_{\partial\Omega}\right)_{\overline\eta}(\widetilde\xi-f\csc\theta\,\overline\nu,\widetilde\xi-f\csc\theta\,\overline\nu) \nonumber \\
&=& \sin\theta\left(\emph{II}_{\partial\Omega}\right)_{\overline\eta}(\widetilde\xi,\widetilde\xi)-2f\left(\emph{II}_{\partial\Omega}\right)_{\overline\eta}(\widetilde\xi,\overline\nu)+f^2\csc\theta\left(\emph{II}_{\partial\Omega}\right)_{\overline\eta}(\overline\nu,\overline\nu) \label{0088}
.\end{eqnarray}

In order to calculate the second term of \eqref{0071}, one needs to calculate $\overline\nabla_\xi\eta$, $\overline\nabla_\xi\nu$ and $\overline\nabla_\xi\overline\nu$. Let $\{e_1,...,e_n\}$ be an orthonormal basis of $T_p\Sigma$ for some point $p \in \Sigma$ such that $[e_i,e_j]=[e_i,\xi]=0$ for $i \neq j$ and extend them via $(\varphi_t)_*$ at $p$. Since $\left<e_i,\eta\right>=0$ and $\left[e_i,\xi\right]=0$, we have
\begin{eqnarray}
\overline\nabla_\xi \eta &=& \left<\overline\nabla_\xi \eta,\eta\right>\eta+\sum_{i=1}^n \left<\overline\nabla_\xi \eta,e_i\right>e_i \nonumber \\
&=& -\sum_{i=1}^n \left<\overline\nabla_\xi e_i,\eta\right>e_i \nonumber \\
&=& -\sum_{i=1}^n \left<\overline\nabla_{e_i}\xi,\eta\right>e_i \nonumber \\
&=& -\sum_{i=1}^n \left<\overline\nabla_{e_i} \left(\xi^\top+f\eta\right),\eta\right>e_i \nonumber \\
&=& -\sum_{i=1}^n \left<\overline\nabla_{e_i}\xi^\top+f\overline\nabla_{e_i}\eta+e_if\,\eta,\eta\right>e_i \nonumber \\
&=& \sum_{i=1}^n \left(\left<\overline\nabla_{e_i}\eta,\xi^\top\right>-e_if\right)e_i \nonumber \\
&=& -\sum_{i=1}^n \left<-\overline\nabla_{\xi^\top}\eta,e_i\right>e_i-\nabla f \nonumber \\
&=& -A\xi^\top-\nabla f \label{0080}
,\end{eqnarray}
where $A$ is the shape operator of $\varphi$. Therefore, if $\{\widetilde{e}_1,...,\widetilde{e}_{n-1}\}$ denotes an orthonormal basis of $T_p(\partial\Sigma)$ with $p \in \partial\Sigma$, extending them via $\left(\varphi_t\right)_*$ at $p$ and using \eqref{0080}, we have
\begin{eqnarray}
\overline\nabla_\xi \nu &=& \left<\overline\nabla_\xi \nu,\nu\right>\nu+\left<\overline\nabla_\xi \nu,\eta\right>\eta+\sum_{i=1}^{n-1} \left<\overline\nabla_\xi \nu,\widetilde{e}_i\right>\widetilde{e}_i \nonumber \\
&=& -\left<\overline\nabla_\xi \eta,\nu\right>\eta-\sum_{i=1}^{n-1} \left<\overline\nabla_\xi \widetilde{e}_i,\nu\right>\widetilde{e}_i \nonumber \\
&=& \left<A\xi^\top+\nabla f,\nu\right>\eta-\sum_{i=1}^{n-1} \left<\overline\nabla_{\widetilde{e}_i}\xi,\nu\right>\widetilde{e}_i \nonumber \\
&=& \left(\frac{\partial f}{\partial\nu}+\left(\emph{II}_\Sigma\right)_\eta(\xi^\top,\nu)\right)\eta-\sum_{i=1}^{n-1}\left<\overline\nabla_{\widetilde{e}_i}\left(\widetilde\xi-f\cot\theta\,\nu+f\eta\right),\nu\right>\widetilde{e}_i \nonumber \\
&=& \left(\frac{\partial f}{\partial\nu}+\left(\emph{II}_\Sigma\right)_\eta(\xi^\top,\nu)\right)\eta-\sum_{i=1}^{n-1}\left(\left<\overline\nabla_{\widetilde{e}_i}\widetilde\xi,\nu\right>-\cot\theta\,\widetilde{e}_if+f\left<\overline\nabla_{\widetilde{e}_i}\eta,\nu\right>\right)\widetilde{e}_i \nonumber \\
&=& \left(\frac{\partial f}{\partial\nu}+\left(\emph{II}_\Sigma\right)_\eta(\xi^\top,\nu)\right)\eta-\sum_{i=1}^{n-1}\left(\left<\nabla_{\widetilde{e}_i}\widetilde\xi,\nu\right>-\cot\theta\,\widetilde{e}_if-f\left<\overline\nabla_{\widetilde{e}_i}\nu,\eta\right>\right)\widetilde{e}_i \nonumber \\
&=& \left(\frac{\partial f}{\partial\nu}+\left(\emph{II}_\Sigma\right)_\eta(\xi^\top,\nu)\right)\eta-A_{\partial\Sigma}\widetilde\xi+\cot\theta\,\widetilde\nabla f+f\left(A\nu-\left(\emph{II}_\Sigma\right)_\eta(\nu,\nu)\nu\right) \label{0081}
,\end{eqnarray}
where $A_{\partial\Sigma}$ is the shape operator of $\iota_{\partial\Sigma} : \partial\Sigma \hookrightarrow \Sigma$ and $\widetilde\nabla f$ is the gradient of $f$ in $\partial\Sigma$. We also have
\begin{eqnarray}
\overline\nabla_\xi\overline{\nu} &=& \left<\overline\nabla_\xi\overline{\nu},\overline\nu\right>\overline\nu+\left<\overline\nabla_\xi\overline{\nu},\overline\eta\right>\overline\eta+\sum_{i=1}^{n-1}\left<\overline\nabla_\xi\overline{\nu},\widetilde{e}_i\right>\widetilde{e}_i \nonumber \\
&=& \left(\emph{II}_{\partial\Omega}\right)_{\overline\eta}(\xi,\overline\nu)\,\overline\eta-\sum_{i=1}^{n-1}\left<\overline\nabla_\xi\widetilde{e}_i,\overline\nu\right>\widetilde{e}_i \nonumber\\
&=& \left(\emph{II}_{\partial\Omega}\right)_{\overline\eta}(\xi,\overline\nu)\,\overline\eta-\sum_{i=1}^{n-1}\left<\overline\nabla_{\widetilde{e}_i}\left(\widetilde\xi-f\csc\theta\,\overline\nu\right),\overline\nu\right>\widetilde{e}_i \nonumber\\
&=& \left(\emph{II}_{\partial\Omega}\right)_{\overline\eta}(\xi,\overline\nu)\,\overline\eta-\sum_{i=1}^{n-1}\left(\left<\overline\nabla_{\widetilde{e}_i}\widetilde\xi,\overline\nu\right>-\csc\theta\,\widetilde{e}_if\right)\widetilde{e}_i \nonumber\\
&=& \left(\emph{II}_{\partial\Omega}\right)_{\overline\eta}(\xi,\overline\nu)\,\overline\eta-\widetilde{A}\widetilde{\xi}-\csc\theta\,\widetilde\nabla f \label{0082}
,\end{eqnarray}
where $\widetilde{A}$ is the shape operator of $\varphi\vert_{\partial\Sigma} : \partial\Sigma \rightarrow \partial\Omega$.

Using \eqref{0081} and \eqref{0082} in \eqref{0070} we obtain
\begin{multline}\label{0083}
\left<\xi,\overline\nabla_\xi\nu-\cos\theta\,\overline\nabla_\xi\overline\nu\right>=f\left(\frac{\partial f}{\partial\nu}+\left(\emph{II}_\Sigma\right)_\eta(\xi^\top,\nu)\right)-\left<A_{\partial\Sigma}\widetilde\xi,\xi\right>+f\left<A\nu,\xi\right>+\\+f^2\cot\theta\left(\emph{II}_\Sigma\right)_\eta(\nu,\nu)+\cos\theta\left<\widetilde{A}\widetilde\xi,\xi\right>
.\end{multline}
Since
\begin{eqnarray}
\left(\emph{II}_\Sigma\right)_\eta(\xi^\top,\nu) &=& \left<\overline\nabla_{\xi^\top}\nu,\eta\right>=\left<\overline\nabla_{\widetilde\xi}\nu,\eta\right>-f\cot\theta\left<\overline\nabla_\nu\nu,\eta\right> \nonumber \\
&=& \left<\overline\nabla_{\widetilde\xi}\left(\cos\theta\,\overline\nu+\sin\theta\,\overline\eta\right),-\sin\theta\,\overline\nu+\cos\theta\,\overline\eta\right>-f\cot\theta\left(\emph{II}_\Sigma\right)_\eta(\nu,\nu) \nonumber\\
&=& \left<\overline\nabla_{\widetilde\xi}\overline\nu,\overline\eta\right>-f\cot\theta\left(\emph{II}_\Sigma\right)_\eta(\nu,\nu) \nonumber\\
&=& \left(\emph{II}_{\partial\Omega}\right)_{\overline\eta}(\widetilde\xi,\overline\nu)-f\cot\theta\left(\emph{II}_\Sigma\right)_\eta(\nu,\nu) \label{0084}
,\end{eqnarray}
\begin{eqnarray}
\left<A_{\partial\Sigma}\widetilde\xi,\xi\right> &=& \left<A_{\partial\Sigma}\widetilde\xi,\widetilde\xi\right>=-\left<\nabla_{\widetilde\xi}\nu,\widetilde\xi\right>\nonumber\\
&=& -\left<\overline\nabla_{\widetilde\xi}\nu-\left(\emph{II}_\Sigma\right)_\eta(\widetilde\xi,\nu)\,\eta,\widetilde\xi\right> \nonumber \\
&=& -\left<\overline\nabla_{\widetilde\xi}\left(\cos\theta\,\overline\nu+\sin\theta\,\overline\eta\right),\widetilde\xi\right> \nonumber\\
&=& -\cos\theta\left<\overline\nabla_{\widetilde\xi}\overline\nu,\widetilde\xi\right>-\sin\theta\left<\overline\nabla_{\widetilde\xi}\overline\eta,\widetilde\xi\right> \nonumber\\
&=& \cos\theta\left<\widetilde{A}\widetilde\xi,\widetilde\xi\right>+\sin\theta\left(\emph{II}_{\partial\Omega}\right)_{\overline\eta}(\widetilde\xi,\widetilde\xi) \nonumber\\
&=& \cos\theta\left<\widetilde{A}\widetilde\xi,\xi\right>+\sin\theta\left(\emph{II}_{\partial\Omega}\right)_{\overline\eta}(\widetilde\xi,\widetilde\xi) \label{0085}
\end{eqnarray}
and
\begin{equation}\label{0086}
\left<A\nu,\xi\right>=\left<A\nu,\xi^\top\right>=\left(\emph{II}_\Sigma\right)_\eta(\xi^\top,\nu)
,\end{equation}
the equations \eqref{0084}, \eqref{0085} and \eqref{0086} in \eqref{0083} yield to
\begin{eqnarray}
\left<\xi,\overline\nabla_\xi\nu-\cos\theta\,\overline\nabla_\xi\overline\nu\right> &=& f\frac{\partial f}{\partial\nu}+2f\left(\emph{II}_\Sigma\right)_\eta(\xi^\top,\nu)+f^2\cot\theta\left(\emph{II}_\Sigma\right)_\eta(\nu,\nu)-\sin\theta\left(\emph{II}_{\partial\Omega}\right)_{\overline\eta}(\widetilde\xi,\widetilde\xi) \nonumber\\
&=& f\frac{\partial f}{\partial\nu}+2f\left(\emph{II}_{\partial\Omega}\right)_{\overline\eta}(\widetilde\xi,\overline\nu)-f^2\cot\theta\left(\emph{II}_\Sigma\right)_\eta(\nu,\nu)-\sin\theta\left(\emph{II}_{\partial\Omega}\right)_{\overline\eta}(\widetilde\xi,\widetilde\xi) \label{0087}
.\end{eqnarray}
Summing \eqref{0088} and \eqref{0087} we obtain
\begin{multline}\label{0089}
\left<\overline\nabla_\xi\xi,\nu-\cos\theta\,\overline\nu\right>+\left<\xi,\overline\nabla_\xi\nu-\cos\theta\,\overline\nabla_\xi\overline\nu\right>=f\left(\frac{\partial f}{\partial\nu}+\left(\csc\theta\left(\emph{II}_{\partial\Omega}\right)_{\overline\eta}(\overline\nu,\overline\nu)-\cot\theta\left(\emph{II}_\Sigma\right)_\eta(\nu,\nu)\right)f\right).
\end{multline}
We complete the proof after integrating \eqref{0089} over $\partial\Omega$.

\section{Proof of Proposition \ref{004.6}}
In this appendix we give a proof of Proposition \ref{004.6} based on the arguments of \cite[Appendix A]{guo2022stable}. First notice that a totally umbilical hypersurface $\varphi : \Sigma \rightarrow \Omega \subseteq \mathbb{M}^{n+1}(c)$ supported on a totally umbilical hypersurface $\partial\Omega$ is $r$-stable if and only if it is $0$-stable. In fact, if $\kappa \in (0,\infty)$ is the umbilicity factor of $\varphi$ then $P_r=\frac{n-r}{n}\binom{n}{r}\kappa^r\,I$. Hence, $S_r=\frac{1}{n-r}\tr P_r=\binom{n}{r}\kappa^r$ and for all $f_1,f_2 \in H^1(\Sigma)$,
\begin{eqnarray*}
\mathcal{I}_{r,\theta}(f_1,f_2) &=& \frac{n-r}{n}\binom{n}{r}\kappa^r\left(\int_\Sigma \left<\nabla f_1,\nabla f_2\right>-\left(\vert{A}\vert^2+nc\right)f_1f_2\,d\mu_\Sigma+\int_{\partial\Sigma}\alpha_\theta f_1f_2\,d\mu_{\partial\Sigma}\right) \\
&=& \frac{n-r}{n}\binom{n}{r}\kappa^r\,\mathcal{I}_{0,\theta}(f_1,f_2),
\end{eqnarray*}
proving the claim.

For the rest of the appendix we will consider the warped product model for $\mathbb{M}^{n+1}(c)$ described on Section 7. Now for $\rho_0 \in (0,R_c)$, consider the Robin eigenvalue problem for $-\Delta-nc$ on a geodesic ball $B_{\rho_0}$ of $\mathbb{M}^n(c)$
\begin{equation}\label{0061}
\begin{dcases*}
-\Delta f-ncf=\lambda f,& in $B_{\rho_0}$ \\
\frac{\partial f}{\partial\overline\eta}-\frac{\cn_c(\rho_0)}{\sn_c(\rho_0)}f=0,& in $\partial B_{\rho_0}$
\end{dcases*}
.\end{equation}
where $\overline\eta$ is the outward unit normal field of $\partial B_{\rho_0}$ and denote $\lambda_1$ and $\lambda_2$ the first and the second Robin eigenvalues of \eqref{0061}. The following lemma, which is adapted from \cite[Proposition A.1]{guo2022stable}, shows that the second eigenvalue is equal to zero.

\begin{lema}\label{00B.1}
$\lambda_2=0$ and its multiplicity is equal to $n$.
\end{lema}

\begin{proof}
Since $B_{\rho_0}$ is rotationally symmetric, the eigenvalues of \eqref{0061} are given by $\tau_{k,l}$, for $(k,l) \in \mathbb{N} \times \left(\mathbb{N}\times\{0\}\right)$, where $\tau_{k,l}$ is the $k$-th eigenvalue of the problem
\begin{equation}\label{0073}
\begin{dcases*}
-f^{\prime\prime}(\rho)-(n-1)\frac{\cn_c(\rho)}{\sn_c(\rho)}f^\prime(\rho)-\left(nc-\frac{l\left(l+n-2\right)}{\sn_c^2(\rho)}\right)f(\rho)=\tau_{\cdot,l}f(\rho),& in $(0,\rho_0)$ \\
f^\prime(\rho_0)-\frac{\cn_c(\rho_0)}{\sn_c(\rho_0)}f(\rho_0)=0
\end{dcases*}
\end{equation}
and $f(0)>0$, $f^\prime(0)>0$ for $l=0$ or $f(\rho) \sim \rho^l$ as $\rho\rightarrow 0$ for $l \in \mathbb{N}$. The eigenvalues $\tau_{k,l}$ of \eqref{0073} are given by
\begin{equation}\label{0079}
\tau_{1,l}=\inf\left\{\dfrac{\int_0^{\rho_0}\left(f^\prime(\rho)^2-\left(nc-\frac{l\left(l+n-2\right)}{\sn_c^2(\rho)}\right)f(\rho)^2\right)\sn_c^{n-1}(\rho)\,d\rho-\sn_c^{n-2}(\rho_0)\cn_c(\rho_0)f(\rho_0)^2}{\int_0^{\rho_0}f(\rho)^2\sn_c^{n-1}(\rho)\,d\rho}\right\}
,\end{equation}
over all $f \in H^1([0,\rho_0])\backslash\{0\}$ with $f(0)=0$ and for $k \in \mathbb{N}$, $\tau_{k,l}$ is defined in a similar way as $\eqref{0079}$ over all $f \in E_{\tau_{m,l}}^\perp\backslash\{0\}$ for all $m \in \{0,...,k\}$, where $E_{\tau_{m,l}}$ is the eigenspace of $\eqref{0073}$ associated to $\tau_{m,l}$ and with $f(0)=0$. Since $\tau_{1,l_1}<\tau_{1,l_2}$ for all $l_1<l_2$ we have that $\lambda_1=\tau_{1,0}$ and $\lambda_2=\min\left\{\tau_{2,0},\tau_{1,1}\right\}$.

Notice that $\tau_{1,1}=0$, since for the positive function $f_1=\sn_c$ in $[0,\rho_0]$, we have $f_1^\prime(\rho)^2+cf_1(\rho)^2=1$, $f_1^{\prime\prime}(\rho)=-cf_1(\rho)$ and
\begin{eqnarray*}
-f_1^{\prime\prime}(\rho)-(n-1)\frac{\cn_c(\rho)}{\sn_c(\rho)}f_1^\prime(\rho)-\left(nc-\frac{n-1}{\sn_c^2(\rho)}\right)f_1(\rho) &=& cf_1(\rho)-(n-1)\frac{f_1^\prime(\rho)^2}{f_1(\rho)}-ncf_1(\rho)+\frac{n-1}{f_1(\rho)} \\
&=& \frac{n-1}{f_1(\rho)} \cdot \left(1-f_1^\prime(\rho)^2\right)-c(n-1)f_1(\rho) \\
&=& 0
\end{eqnarray*}
for all $\rho\in(0,\rho_0)$. The same argument used in \cite[Proposition A.1]{guo2022stable} shows that $\tau_{2,0}>0$ and its multiplicity is the same as that of second eigenvalue $n-1$ of the Laplacian of $\mathbb{S}^{n-1}$.
\end{proof}

\begin{proof}[Proof of Proposition \ref{004.6}]
Let $\varphi : \Sigma \rightarrow \Omega \subseteq \mathbb{M}^{n+1}(c)$ be a totally umbilical capillary hypersurface supported on $\partial\Omega$ with contact angle $\theta\in(0,\pi)$. If $\kappa$ and denotes the principal curvature of $\varphi$, the Gauss equation shows the sectional curvature is equal to $K=K_{\mathbb{M}^{n+1}(c)}+\kappa^2=\kappa^2+nc$. From \eqref{0092} we have that $\nu=-\cot\theta\,\eta+\csc\theta\,\overline\eta$, thus for all $X,Y \in \Gamma(T\Sigma)$
\begin{eqnarray*}
\left(\emph{II}_{\partial\Sigma}\right)_\nu(X,Y) &=& \left<\nabla_XY,\nu\right> \nonumber \\
&=& \left(\emph{II}_{\partial\Omega}\right)_{\overline\eta}(X,Y)\,\csc\theta-\left(\emph{II}_\Sigma\right)_\eta(X,Y)\,\cot\theta
,\end{eqnarray*}
proving that $\iota_{\partial\Sigma} : \partial\Sigma \hookrightarrow \Sigma$ is a totally umbilical hypersurface with umbilicity factor equal to $\kappa_{\partial\Sigma}=\kappa_{\partial\Omega}\csc\theta-\kappa\cot\theta$. Thus the bilinear form associated to the weak formulation of \eqref{0061} is exactly the $0$-index form $\mathcal{I}_{0,\theta}$ described in \eqref{0010}. From the Lemma \ref{00B.1}, the Morse index of $\mathcal{I}_{0,\theta}$ over $H^1(\Sigma)$ is equal to $1$. But since the function
\begin{equation}\label{0091}
f(\rho)=\begin{dcases*}
\frac{\cn_c(\rho_0)\cn_c(\rho)-1}{nc},& $c \neq 0$ \\ -\frac{\rho^2+\rho_0^2}{2n},& $c=0$
\end{dcases*}
\end{equation}
is a non-positive function in $\Sigma$ such that, for $c \neq 0$,
\begin{eqnarray*}
\Delta f+ncf &=& \frac{1}{\sn_c^{n-1}(\rho)}\frac{d}{d\rho}\left(\sn_c^{n-1}(\rho)\frac{d}{d\rho}\left(\frac{\cn_c(\rho_0)\cn_c(\rho)-1}{nc}\right)\right)+\cn_c(\rho_0)\cn_c(\rho)-1 \\
&=& -\frac{\cn_c(\rho_0)}{n\sn_c^{n-1}(\rho)}\frac{d}{d\rho}\left(\sn_c^n(\rho)\right)+\cn_c(\rho_0)\cn_c(\rho)-1 \\
&=& -\frac{cn_c(\rho_0)}{n\sn_c^{n-1}(\rho)} \cdot n\sn_c^{n-1}(\rho)\cdot\cn_c(\rho)+\cn_c(\rho_0)\cn_c(\rho)-1=-1
\end{eqnarray*}
in $\Sigma$ and
\begin{eqnarray*}
\frac{\partial f}{\partial\overline\eta} &=& \left.\frac{d}{d\rho}\left(\frac{\cn_c(\rho_0)\cn_c(\rho)-1}{nc}\right)\right\vert_{\rho=\rho_0}=-\frac{\cn_c(\rho_0)}{n} \cdot\sn_c(\rho_0)=-\frac{\cn_c(\rho_0)}{\sn_c(\rho_0)} \cdot \frac{c\sn_c^2(\rho_0)}{nc} \\
&=& \frac{\cn_c(\rho_0)}{\sn_c(\rho_0)}\cdot\frac{\cn_c^2(\rho_0)-1}{nc}=\frac{\cn_c(\rho_0)}{\sn_c(\rho_0)}f(\rho_0)
\end{eqnarray*}
on $\partial\Sigma$, and for $c=0$,
\begin{equation*}
\Delta f=\frac{1}{\rho^{n-1}}\frac{d}{d\rho}\left(\rho^{n-1}\frac{d}{d\rho}\left(-\frac{\rho^2+\rho_0^2}{2n}\right)\right)=-\frac{1}{n\rho^{n-1}}\frac{d}{d\rho}\left(\rho^n\right)=-1
\end{equation*}
in $\Sigma$ and
\begin{equation*}
\frac{\partial f}{\partial\overline\eta}=\left.\frac{d}{d\rho}\left(-\frac{\rho^2+\rho_0^2}{2n}\right)\right\vert_{\rho=\rho_0}=-\frac{\rho_0}{n}=-\frac{1}{\rho_0}\cdot\frac{2\rho_0^2}{2n}=\frac{\cn_0(\rho_0)}{\sn_0(\rho_0)}f(\rho_0)
\end{equation*}
on $\partial\Sigma$, \cite[Proposition 3.1]{tran2020morse} implies the Morse index of $\mathcal{I}_{0,\theta}$ over $\mathcal{F}$ is equal to $1-1=0$, proving that $\varphi$ is $0$-stable.
\end{proof}

\section{Proof of Lemma \ref{012.1}}
In this appendix we give a proof of Lemma \ref{012.1}.
\begin{enumerate}[(i)]
\item The first claim is immediate.

\item Now assume that $\lambda_1<0$ and let $f_i \in E_{\lambda_i}$ be an eigenfunction associated to the $i$-th eigenvalue $\lambda_i$ of $T_r^S$ such that $\left\{f_i\right\}_{i \in \mathbb{N}}$ is an orthonormal basis for $L^2(\Sigma)$. Since $f_1$ does not change its sign, we have $\int_\Sigma f_1\,d\mu_\Sigma \neq 0$. Given $f \in H_0^1(\Sigma)$, let $a=-\frac{\int_\Sigma f\,d\mu_\Sigma}{\int_\Sigma f_1\,d\mu_\Sigma}$ and define $\widetilde{f}=af_1+f$; we have $\int_\Sigma \widetilde{f}\,d\mu_\Sigma=0$. If $0$ is not an eigenvalue of $T_r^S$, it follows from the Fredholm alternative that there exists a unique function $f \in H_0^1(\Sigma)$ satisfying $T_r^Sf=-1$ weakly in $\Sigma$, i.e., $I_r^S(f,g)=-\int_\Sigma g\,d\mu_\Sigma$ for all $g \in H_0^1(\Sigma)$. Therefore,
\begin{eqnarray*}
I_r^S(\widetilde{f},\widetilde{f}) &=& I_r^S(af_1+f,af_1+f)=a^2I_r^S(f_1,f_1)+2aI_r^S(f_1,f)+I_r^S(f,f) \\
&=& a^2\lambda_1-2a\int_\Sigma f_1\,d\mu_\Sigma-\int_\Sigma f\,d\mu_\Sigma \\
&=& a^2\lambda_1+\int_\Sigma f\,d\mu_\Sigma,
\end{eqnarray*}
proving that $\varphi$ is symmetric $r$-unstable if $\int_\Sigma f\,d\mu_\Sigma<0$. If $\int_\Sigma f\,d\mu_\Sigma \geq 0$ we have
\begin{equation}\label{0060}
I_r^S(f,f)=-\int_\Sigma f\,d\mu_\Sigma \leq 0.
\end{equation}
Since $\lambda_2>0$, we have $f \notin E_{\lambda_1}^\perp$; otherwise we would obtain as a consequence of \eqref{0058} and \eqref{0060} that 
\begin{equation}\label{desigualdade} 0 \geq I_r^S\left(\frac{f}{\left\Vert{f}\right\Vert_{L^2(\Sigma)}},\frac{f}{\left\Vert{f}\right\Vert_{L^2(\Sigma)}}\right) \geq \lambda_2>0, 
\end{equation}
which is a contradiction. Thus, given $\overline{f} \in \mathcal{F}$, there exist a unique $b \in \mathbb{R}$ and $\widehat{f} \in E_{\lambda_1}^\perp$ such that $\overline{f}=bf+\widehat{f}$. Therefore,
\begin{eqnarray*}
I_r^S(\overline{f},\overline{f}) &=& bI_r^S(f,\overline{f})+I_r^S(\widehat{f},bf+\widehat{f}) \\
&=& -b\int_\Sigma\overline{f}\,d\mu_\Sigma-b\int_\Sigma\widehat{f}\,d\mu_\Sigma+I_r^S(\widehat{f},\widehat{f}) \\
&\geq& b^2\int_\Sigma f\,d\mu_\Sigma+\lambda_2\int_\Sigma \widehat{f}^2\,d\mu_\Sigma \geq 0,
\end{eqnarray*}
which implies that $\varphi$ is symmetric $r$-stable and proving (ii).

\item In order to prove (iii), suppose $f=f_2$ and let $\widetilde{f}=af_1+f_2$. It follows from the orthonormality of $\left\{f_l\right\}_{l \in \mathbb{N}}$ that
\begin{eqnarray*}
I_r^S(\widetilde{f},\widetilde{f}) &=& a^2I_r^S(f_1,f_1)+2aI_r^S(f_1,f_2)+I_r^S(f_2,f_2) \\
&=& a^2\lambda_1\int_\Sigma f_1^2\,d\mu_\Sigma+2a\lambda_2\int_\Sigma f_1f_2\,d\mu_\Sigma+\lambda_2\int_\Sigma f_2^2\,d\mu_\Sigma \\
&=& \lambda_1\left(\int_\Sigma f_2\,d\mu_\Sigma\right)^2\left(\int_\Sigma f_1\,d\mu_\Sigma\right)^{-2}+\lambda_2.
\end{eqnarray*}
Therefore, if $\lambda_2 \leq 0$ and $\int_\Sigma f_2\,d\mu_\Sigma \neq 0$, we have $I_r^S(\widetilde{f},\widetilde{f})<0$, proving that $\varphi$ is $r$-unstable and proving the first claim of (iii). The result also holds if $\lambda_2<0$, which is (v).

\item Finally, in order to prove the claim of (iv), notice that since $\lambda_2=0$ and $1 \in E_{\lambda_2}^\perp$ by hypothesis, there exists a unique $f \in H_0^1(\Sigma) \cap E_{\lambda_2}^\perp$ such that $T_r^Sf=-1$ weakly in $\Sigma$.

If $\int_\Sigma f\,d\mu_\Sigma >0$, we use \eqref{0058}, \eqref{0060} and (\ref{desigualdade}) in order to conclude that $f \notin E_{\lambda_1}^\perp$ and we proceed as in the proof of ii). 

If $\int_\Sigma f\,d\mu_\Sigma=0$, we have $$I_r^S(f,f)=-\int_\Sigma f\,d\mu_\Sigma=0=\lambda_2.$$ If $f \in E_{\lambda_1}^\perp$, by (\ref{0058}) $f$ would be an eigenfunction associated to $\lambda_2=0$, thus $T_r^Sf=0$, which is a contradiction. Therefore $f \notin E_{\lambda_1}^\perp$ and the proof is analogous to the last part of the proof of (ii).
\end{enumerate}

\bibliographystyle{acm}
\bibliography{bibliography-articles}
\end{document}